\documentclass[11pt,a4paper]{amsart}
\usepackage{a4wide} 

\usepackage{pdfpages}
 
 \usepackage{amsmath, amssymb, amsfonts, enumerate}
\usepackage[all]{xy}
\usepackage{amscd}
\usepackage{comment}
\usepackage{mathtools}
\usepackage{tikz} 
\usepackage{tikz-cd}

\usepackage{url}
\usepackage{hyperref}

\newtheorem{theoremletter}{Theorem}

\newtheorem{corollaryletter}{Corollary}

\makeatletter
\newtheorem*{rep@theorem}{\rep@title}
\newcommand{\newreptheorem}[2]{%
\newenvironment{rep#1}[1]{%
 \def\rep@title{#2 \ref{##1}}%
 \begin{rep@theorem}}%
 {\end{rep@theorem}}}
\makeatother

\newtheorem{theorem}{Theorem}[section]
\newtheorem*{theorem*}{Theorem}
 \newtheorem*{conjecture*}{Conjecture}
  \newtheorem{conjecture}[theorem]{Conjecture}
  \newtheorem*{corollary*}{Corollary}
  
  \newtheorem{lemma}[theorem]{Lemma}
\newtheorem{corollary}[theorem]{Corollary}
\newtheorem{proposition}[theorem]{Proposition}

\newreptheorem{theorem}{Theorem}
 \newreptheorem{proposition}{Proposition}
 \newreptheorem{corollary}{Corollary} 

 \theoremstyle{definition}
 \newtheorem{definition}[theorem]{Definition} 
 \newtheorem{question}[theorem]{Question}
 \newtheorem{remark}[theorem]{Remark}
  \newtheorem*{setting*}{Setting (P)}

\numberwithin{equation}{section}
 
\newcommand {\N}{\mathbb{N}} 
\newcommand {\Z}{\mathbb{Z}} 
\newcommand {\R}{\mathbb{R}} 
\newcommand {\Q}{\mathbb{Q}} 
\newcommand {\C}{\mathbb{C}} 

\newcommand{\G}{\mathbb{G}}

\newcommand{\DD}{\mathcal{D}}
\newcommand{\EE}{\mathcal{E}}
\newcommand{\FF}{\mathcal{F}}

\newcommand{\OO}{\mathcal{O}}

\newcommand{\VV}{\mathcal{V}}

\newcommand{\XX}{\mathcal{X}}

\newcommand{\Proj}{\mathbb{P}}

\newcommand{\Sch}{\mathrm{Sch}}

\newcommand{\floor}[1]{\left\lfloor #1 \right\rfloor}

\newcommand\numberthis{\addtocounter{equation}{1}\tag{\theequation}}

\DeclareMathOperator{\vol}{vol}

\DeclareMathOperator{\Ker}{Ker}

\DeclareMathOperator{\Mor}{Mor}

\DeclareMathOperator{\im}{Im}

\DeclareMathOperator{\Id}{Id}

\DeclareMathOperator{\Aut}{Aut}

\DeclareMathOperator{\supp}{supp}

\DeclareMathOperator{\rank}{rank}

\DeclareMathOperator{\Hilb}{Hilb} 
\DeclareMathOperator{\Tr}{Tr} 
 
\DeclareMathOperator{\length}{length} 
\DeclareMathOperator{\Gal}{Gal}

\DeclareMathOperator{\Ensemble}{Ens}

\usepackage{setspace}
\begin{document}	
\title[Generalized integral points on abelian varieties]
{Generalized integral points on abelian varieties and the Geometric Lang-Vojta conjecture}
\author[Xuan-Kien Phung]{Xuan Kien Phung}
\address{Département d'informatique et de recherche opérationnelle,  Université de Montréal, Montréal, Québec, H3T 1J4, Canada.}
\address{Département de mathématiques et de statistique, Université de Montréal, Montréal, Québec, H3T 1J4, Canada.} 
\email{phungxuankien1@gmail.com}
\subjclass[2020]{14G05, 11J95, 11D45, 14H05, 14D10, 14H30, 14H55}
\keywords{integral point, geometric Lang-Vojta Conjecture, abelian variety, hyperbolicity, generic emptiness, compact Riemann surface}

\begin{abstract}  
Let $f \colon \mathcal{A} \to B$ be a family of abelian varieties over a compact Riemann surface $B$ and fix an effective horizontal divisor $\mathcal{D} \subset \mathcal{A}$. We study $(S, \mathcal{D})$-integral sections $\sigma$ of the family $\mathcal{A}$ where $S \subset B$ is arbitrary. These sections $\sigma$ are algebraic and satisfy the geometric condition $f(\sigma(B) \cap \mathcal{D})\subset S$. Developing the work of Parshin, we establish new quantitative results concerning the finiteness and the polynomial growth of large unions of $(S, \mathcal{D})$-integral sections where $S$ can vary and is required to be finite only in a thin analytic open subset of~$B$. Such results are out of the range of purely algebraic methods and imply new evidence and interesting phenomena to the Geometric Lang-Vojta conjecture.
\end{abstract}

\maketitle
 
\setcounter{tocdepth}{1}

 \section{Introduction} 
 
 Let $\bar{C}$ be a smooth projective curve defined over an
algebraically closed field $k$ of characteristic zero.  
Let $S$ be a finite set of points of $\bar{C}$ 
and denote $C= \bar{C} \setminus S$.  
Consider a smooth affine variety $X$ of log-general type over $k(\bar{C})$ 
with a model $f\colon \XX \to C$. 
Let $\DD$ be the hyperplane at infinity in 
a compactification $\bar{\XX}$ of $\XX$. 
A weak form of the  \emph{Geometric Lang-Vojta conjecture} 
(cf. \cite[F.3.5]{hindry-silverman-book}, \cite[Conjecture 4.4]{vojta-86}, \cite{lang-86}, cf. also 
\cite{cor-zan-lang-voj-function-field-p2}, \cite{cor-zan-lang-voj-p2-2013})  
implies the following: 
\begin{conjecture} 
[Lang-Vojta] 
\label{conjecture-lang-vojta-geometric} 
There exists a proper closed subset $Z \subset X$ and  $m>0$ with the following property. 
Let $\mathcal{Z}$ be the Zariski closure of $Z$ in $\XX$. 
Then for every section $\sigma \colon C \to \XX$ with $\sigma(C) \not \subset \mathcal{Z}$, 
let $\bar{\sigma} \colon  \bar{C} \to \bar{\XX}$ be the extension of $\sigma$, we have   
\begin{equation}
\label{e:lang-voijta-linear}
\deg_{\bar{C}} \bar{\sigma}^* \DD \leq m (2g(\bar{C})- 2+ \#S). 
\end{equation}
\end{conjecture}

In the strong form of the Geometric Lang-Vojta conjecture, note that  $m$ should be independent of the curve $\bar{C}$.  
The bound \eqref{e:lang-voijta-linear} is  known   
 when $X$ is a curve and $m$ is  effective when $X$ is an elliptic curve (cf. \cite[Corollary 8.5]{hindry-silverman-88}). 
Several results are also known when $\bar{X}= \Proj^2$ in all arithmetic, analytic and algebraic settings 
(e.g. \cite{cor-zan-lang-voj-function-field-p2}, \cite{cor-zan-lang-voj-p2-2013}, 
\cite{turchet-p2}).   
  When $X$ is the complement of an effective ample divisor in an abelian variety $\bar{X}$, 
the exceptional set $Z$ can be taken to be empty 
(by an immediate induction using \cite[Lemma 4.2.1]{vojta-book}) and  Conjecture 
\ref{conjecture-lang-vojta-geometric}  holds when 
the trace $\Tr_{ k(\bar{C} )/ k}(\bar{X})$ 
(cf. \cite{chow-image-trace-1}, \cite{chow-image-trace-2}) of $\bar{X}$ is zero (cf. \cite{buium-94}) 
or when $\bar{X}$ is defined over $k= \C$ (cf. \cite{noguchi-winkelmann-04}). 
Even in these cases, it is particularly  difficult to obtain a purely algebraic proof: 
the proof in \cite{buium-94} uses the Kolchin differential theory while 
 \cite{noguchi-winkelmann-04} is based on the tool of jet-differentials. 
No similar results were known 
for a general abelian variety. 
 \par
The main goal of the paper is to 
provide new phenomena and evidence for Conjecture~\ref{conjecture-lang-vojta-geometric} when 
$\bar{\XX} \to \bar{C}$ 
is a family of abelian varieties 
over a compact Riemann surface  where $S \subset \bar{C}$ can vary and is not necessarily finite 
(Theorem  \ref{t:parshin-generic-emptiness-higher-dimension}, Theorem~\ref{t:main-generalization-of-A},  
Corollary \ref{t:lang-vojta-corollary}, Corollary~\ref{t:lang-vojta-corollary-generalized-1}).
\par 
We develop the work of Parshin in \cite{parshin-90} 
to formulate a   {hyperbolic-homotopic method}  
as a substitute for the classical intersection theory. 
The technique allows us to address quantitative problems 
concerning \emph{$(S,\DD)$-integral points}
, namely, rational points $P \in X(k(\bar{C}))$  whose induced section $\sigma_P \in \mathcal{X}(\bar{C})$ satisfy  $f(\sigma_P(\bar{C}) \cap \mathcal{D})\subset S$. 
The hyperbolic part (Theorem \ref{t:linear-bound-s-base-curve-1}), which may be of independent interest,  
concerns the hyperbolic length of loops on compact Riemann surfaces. 
We obtain 
a reasonable estimation of the growth of integral sections of $\bar{\XX} \to \bar{C}$ 
in the sense that 
it recovers corresponding results in all special known cases  (cf. Section  \ref{r:optimal-growth}). 
Certain results on the topology of the intersection locus of sections with a horizontal divisor 
in an abelian fibration are also established 
(Theorem~\ref{p:uncountable-limit-point}). 

\subsection{Finiteness and growth of integral points} 

In \cite[Theorem 3.2]{parshin-90}, 
Parshin proved the following general finiteness theorem with 
an elegant hyperbolic method: 

\begin{theorem}
[Parshin] 
\label{t:parshin-90}
Let $S \subset B$ be a finite subset. 
Let $f \colon \mathcal{A} \to B$ be a family of abelian varieties and let $\DD \subset \mathcal{A}$ be an effective integral divisor which dominates $B$. 
Assume that $D_K$ does not contain any translates of nonzero abelian subvarieties of $\mathcal{A}_K$. 
Then the set 
$ 
\{ P \in \mathcal{A}_K(K) \colon  f(\sigma_P (B) \cap \DD) \subset S \} 
$ is finite modulo the trace  $\Tr_{K/\C}(\mathcal{A}_K)(\C)$. 
\end{theorem}
\par 
Note that 
 by the Lang-N\' eron theorem  \cite{lang-neron-theorem-paper}, $\mathcal{A}_K(K)/ \Tr_{K/\C}(\mathcal{A}_K)(\C)$ is an abelian group of finite rank. We establish the following generalization of  Theorem~\ref{t:parshin-90} (cf.~Section~\ref{s:parshin-generic-emptiness-higher-dimension}, see also  Theorem~\ref{t:main-generalization-of-A} for an even  stronger statement). 
\begin{setting*}
Let $A$ be an abelian variety over $K$.   
Let $D \subset A$ be a reduced effective ample divisor 
of $A$ which does not contain any translates of nonzero abelian subvarieties (e.g., when $A$ is simple).  
Let $\DD$ be the Zariski closure of $D$ in a proper flat model $f \colon \mathcal{A} \to B$ of $A$.
Let $T \subset B$ be the finite subset 
consisting of $b \in B$ such that $\mathcal{A}_b$ is not smooth. 
\end{setting*} 

\begin{theoremletter}
\label{t:parshin-generic-emptiness-higher-dimension}  
In Setting (P), let $W\supset T$ be any finite union of disjoint closed discs in $B$ such that distinct points of $T$ are contained in distinct discs. 
Let $B_0= B \setminus W$. 
Then there exists $m >0$ such that  for all $s \in \N$, the set $I_s \coloneqq \{ P \in A(K) \colon \# f(\sigma_P (B_0) \cap \DD) \leq s \}$ satisfies 
\begin{equation}
\label{e:qualitative-bound-abelian} 
\# I_s  \text{ mod } \Tr_{K/\C} (A) (\C)
 \leq m(s+1)^{2 \dim A . \rank \pi_1(B_0)}.
\end{equation}  
\end{theoremletter}
Here, $\sigma_P \colon B \to \mathcal{A}$ is the induced section of $P \in A(K)$ and  $\rank \pi_1(B_0)$ denotes the minimal number of generators of 
the finitely generated group $\pi_1(B_0)$. 
\par 
Theorem \ref{t:parshin-generic-emptiness-higher-dimension} implies the polynomial bound in $\# S$ on the number of 
$(S, \DD)$-integral sections proved in \cite{hindry-silverman-88} by Hindry-Silverman for elliptic 
curves or by Buium in \cite{buium-94} when $\Tr_{K /\C}(A)=0$ and in 
\cite{noguchi-winkelmann-04} by Noguchi-Winkelman when $A$ is defined over~$\C$ 
(Section  \ref{r:discussion-poly-bound-intro}). 
When $A$ is an elliptic curve, results in  \cite{phung-19-elliptic} show that 
Theorem \ref{t:parshin-generic-emptiness-higher-dimension} 
can be strengthened where $\DD$ can also vary in a compact family 
of divisors on $\mathcal{A}$ in the definition of union $I_s$.   
\par
To the limit of our knowledge, 
even finiteness results of $(S, \DD)$-integral sections for certain $S \subset B$   
countably infinite is not stated elsewhere before in the literature. 
To obtain such results,  
  establishing a height bound as in traditional approaches, 
which depends on the cardinality of $S$, 
is clearly not sufficient. 
In the case of number fields, 
the closest related finiteness result seems to be a result of Silverman \cite{silverman-87}   (see also \cite[Theorem 1.39]{phung-19-phd}) 
for elliptic curves but it requires some strong restrictions on the set of sections.

\subsection{Application to the Geometric Lang-Vojta conjecture} 
Theorem \ref{t:parshin-generic-emptiness-higher-dimension} 
implies a certain extension of Conjecture \ref{conjecture-lang-vojta-geometric} 
for polarized abelian varieties as follows. 

\begin{corollary} 
\label{t:lang-vojta-corollary} 
In Setting \emph{(P)}, let $W \supset T$ be a finite union of disjoint closed discs in $B$ such that distinct points of $T$ are contained in distinct discs.  
Let $B_0= B \setminus  W$.  
Then for every $s >0$, there exists $M=M(\mathcal{A}, \DD, B_0, s) >0$ such that for every section $\sigma \colon B \to \mathcal{A}$ with 
$\# f(\sigma(B_0) \cap \DD) \leq s$, one has $\deg_B \sigma^* \DD < M$. 
\end{corollary}

When $\Tr_{K/\C}(A)$ is nonzero and $A$ is nonconstant, the above result was not known even 
under the hypothesis $W=T$. Note that we obtain in  Corollary~\ref{t:lang-vojta-corollary-generalized-1} a stronger statement. 

\begin{proof}
We equip $A/K$ with a symmetric ample line bundle 
$L$ 
and consider the corresponding canonical N\' eron-Tate height function 
$\widehat{h}_L \colon A(K)/\Tr_{K / \C}(A)(\C) \to \R_+$. 
Since $L$ is ample, 
there exists $n \in \N^*$ such that the line bundle 
$L^{\otimes n} \otimes \OO(- D)$ is very ample on $A$. 
By standard positivity properties of 
height theory (cf.  \cite{hindry-silverman-book}, \cite{conrad-trace}), 
there exists a finite number $c >0$ 
such that for every $P \in A(K)$ with $\sigma_P \in \mathcal{A}(B)$ 
the corresponding section, we have: 
\begin{equation} 
\label{e:lang-vojta-corollary-general}
n \widehat{h}_L (P) +  c > \deg_B \sigma_P^* \OO(\DD). 
\end{equation}  
\par 
Let $s >0$ and let $I_s \coloneqq \{ P \in A(K) \colon \# f(\sigma_P (B_0) \cap \DD) \leq s \}$.
As $I_s$ mod  $\Tr_{K/ \C}(A)(\C)$ is finite by Theorem \ref{t:parshin-generic-emptiness-higher-dimension} 
and as the canonical height $\widehat{h}_L$ is invariant under translations by the trace, 
we can define a finite number $H \coloneqq  \max_{P \in I_s} \widehat{h}_L(P)$. 
For $\sigma\in I_s$,  
 \eqref{e:lang-vojta-corollary-general} implies that: 
\begin{equation*} 
\deg_B \sigma^* \DD \leq \sup_{\tau \in I_s} \deg_B \tau^* \DD 
\leq n \max_{P \in I_s} \widehat{h}_L(P) + c \leq n H +c.  
\end{equation*} 
\par 
The conclusion follows by setting $M= nH+c >0$.  
\end{proof} 

In particular, the Lefschetz Principle and Corollary \ref{t:lang-vojta-corollary} 
imply immediately that if we allow $m$ to depend on $\# S$, Conjecture \ref{conjecture-lang-vojta-geometric} is true 
for $X = \bar{X} \setminus D$ with moreover an empty exceptional set $Z$, 
where $\bar{X}/ k(\bar{C})$ is an abelian variety and $D \subset \bar{X}$ 
is any effective ample divisor not containing 
any translates of nonzero abelian subvarieties of $\bar{X}$.

\subsection{Hyperbolic length on Riemann surfaces} 
One of the new ingredients in the proof of 
Theorem \ref{t:parshin-generic-emptiness-higher-dimension} 
is the following linear bound 
on the hyperbolic length of loops in 
various complements of a Riemann surface. Let $U$ be a finite union of disjoint closed discs in the Riemann surface 
$B$ and denote $B_0 \coloneqq B \setminus U$. 
We prove the following  estimation  (Section \ref{proof-of-hyper-linear-bound-chapter}):  
 
\begin{theoremletter}
\label{t:linear-bound-s-base-curve-1}
Let $\alpha \in \pi_1(B_0) \setminus \{0\}$. 
Then there exists $L >0$ with the following property. 
For any finite subset $S \subset B_0$,  
there exists  a piecewise smooth loop  
$\gamma \subset B_0 \setminus S$ which represents     
the free homotopy class $\alpha$ in $B_0$ and satisfies $ 
\length_{d_{B_0 \setminus S}} (\gamma) \leq L(\#S + 1)$. 
\end{theoremletter} 

\par 
Here, $d_{B_0 \setminus S}$ denotes the intrinsic Kobayashi hyperbolic metric on $B_0 \setminus S$ 
(Definition \ref{pseudo Kobayashi hyperbolic metric}). More geometric information on the loop $\gamma$ is given in  Theorem \ref{t:collars-for-linear-hyperbolic-bound} and Theorem~\ref{t:collar-moving-discs-general}.  
We sketch briefly below the role of 
Theorem \ref{t:linear-bound-s-base-curve-1} in the proof of Theorem \ref{t:parshin-generic-emptiness-higher-dimension}. 
\par  
In Setting(P), let  $U$ be a certain finite union of disjoint closed discs in $B$ and let $B_0 \coloneqq B \setminus U$. 
For every finite subset $S \subset B_0$, 
the image of a holomorphic section $ 
\sigma \colon B_0 \setminus S \to (\mathcal{A} \setminus \DD)\vert_{B_0 \setminus S}$  
is a totally geodesic subspace with respect to the Kobayashi hyperbolic metrics 
$d_{B_0 \setminus S}$ and $d_{(\mathcal{A} \setminus \DD)\vert_{B_0 \setminus S}}$. 
In particular, 
$\length_{d_{(\mathcal{A} \setminus \DD)\vert_{B_0 \setminus S}}} 
\sigma (\gamma) = \length_{d_{B_0 \setminus S}} (\gamma)$ 
 for every loop $\gamma \subset B_0 \setminus S$. 
 By Ehresmann's theorem, 
 $\mathcal{A}_{B_0} \to B_0$ is a fiber bundle in the differential category. 
 Let $w_0 \in \mathcal{A}_{B_0}$ and $b_0= f(w_0) \in  B_0$. 
Every algebraic section $\sigma \colon B_0 \to \mathcal{A}_{B_0}$ induces 
a homotopy section $i_\sigma$ of the short exact sequence:   
\begin{equation*}
0 \to \pi_1(\mathcal{A}_{b_0}, w_0) \to \pi_1(\mathcal{A}_{B_0}, w_0) \xrightarrow{f_*} \pi_1(B_0,b_0) \to 0.  
\end{equation*}
\par 
A   reduction step of Parshin says that it is enough to bound  
the number of homotopy sections $i_\sigma$ 
in order to bound the number of algebraic sections $\sigma$.  
Fix a system of generators $\alpha_1, \dots, \alpha_k$ of $\pi_1(B_0, b_0)$. 
A theorem of Green says that $(\mathcal{A}\setminus \DD)_{B_0}$ is hyperbolically embedded 
in $\mathcal{A}$. 
Therefore, the number of homotopy sections will be controlled if we can bound 
$\length_{d_{(\mathcal{A} \setminus \DD)\vert_{B_0}}} \sigma (\gamma_i) = \length_{d_{B_0}} (\gamma_i)$  
for some representative loop $\gamma_i$ of each $\alpha_i$. 
\par
An $(S \cup U, \DD)$-integral section $\sigma \colon B \to \mathcal{A}$ 
does not induce a section $B_0 \to \mathcal{A}_{B_0}$ in general 
but only induces a section $B_0 \setminus S \to (\mathcal{A} \setminus \DD)\vert_{B_0 \setminus S}$. 
However, as $d_{B_0 \setminus S} \geq d_{B_0}\vert_{B_0 \setminus S}$,  it suffices 
to bound $\length_{d_{B_0 \setminus S}} (\gamma_i)$ for some loop $\gamma_i \subset B_0 \setminus S$ 
representing $\alpha_i$ for $i= 1, \dots, k$. 
Therefore, 
Theorem \ref{t:linear-bound-s-base-curve-1} plays a crucial role 
for the quantitative estimation 
of $I_s$ in Theorem~\ref{t:parshin-generic-emptiness-higher-dimension}. 
\par 
We  investigate some further related aspects of Theorem 
\ref{t:linear-bound-s-base-curve-1}.  
For each free homotopy class $\alpha \in \pi_1(B_0)$ and each $s \in \N$, 
we can associate a constant 
\begin{equation}
L(\alpha, s) \coloneqq \sup_{\# S \leq s} \inf_{[\gamma]= \alpha} \length_{d_{B_0 \setminus S}} (\gamma) \in \R_+, 
\end{equation}
where $S$ runs over all subsets of $B_0$ of cardinality at most $s$ and 
$\gamma$ runs over all loops in $B_0 \setminus S$ 
which represent $\alpha$. 
Theorem \ref{t:linear-bound-s-base-curve-1} asserts that 
$L(\alpha, s)$ grows at most linearly in $s$. 
Moreover, we prove the following optimal lower bound (cf. Section \ref{proof-of-hyper-lower-bound-chapter}, Remark \ref{r:optimal-lower-bound-s-log-s} for the optimality): 

\begin{theoremletter}
\label{t:linear-bound-sqrt-s-log-s-base-curve-1} 
Given $\alpha \in \pi_1(B_0) \setminus \{0\}$, there exists 
$c >0$ such that for every $s \in \N$:   
\begin{equation}
\label{e:sqrt-log-hyper-length-bound-in-chapter} 
L(\alpha , s) \geq \frac{cs^{1/2}}{\ln (s+2)}. 
\end{equation}
\end{theoremletter} 

It would be interesting to understand the asymptotic  
behavior of $L(\alpha, s)$ in terms of $s$: 

\begin{question}
What are the limits: 
\begin{align}
\deg^-(\alpha)\coloneqq 
\liminf_{s \to \infty} \frac{\ln L(\alpha, s)}{\ln s}, 
\quad \deg^+(\alpha)   \coloneqq \limsup_{s \to \infty} \frac{\ln L(\alpha, s)}{\ln s},   
\end{align}
which correspond to the lower and upper polynomial growth degrees  
of $L(\alpha, s)$ in terms of $s$? 
\end{question}
By Theorem \ref{t:linear-bound-s-base-curve-1} and Theorem \ref{t:linear-bound-sqrt-s-log-s-base-curve-1}, 
we know that for every $\alpha \in \pi_1(B_0) \setminus \{0\}$, we have: 
\begin{equation}
1/2 \leq \deg^-(\alpha) \leq \deg^+(\alpha) \leq 1. 
\end{equation}

If we require $\alpha$ to belong to a certain base of $\pi_1(B_0)$ (cf. Section~ \ref{l:primitive-simple-basis-fundamental-group}), 
the constant $L>0$ in Theorem \ref{t:linear-bound-s-base-curve-1} 
depends only on $U$ and the Riemann surface $B$ 
(cf. Theorem \ref{t:collars-for-linear-hyperbolic-bound}). \par 

\section{Further applications and remarks}

\subsection{Intersection locus of sections with divisors} 

In Setting (P), assume in addition that $\Tr_{K/ \C}(A)=0$. 
For every subset $R \subset A(K)\setminus D$, we define the intersection locus:    
\begin{equation*}
I(R,\DD) \coloneqq \cup_{P \in R} f(\sigma_P(B) \cap \DD) \subset B.
\end{equation*}

\par 
We prove the following result which concerns the topology of the intersection of sections of an abelian scheme with a horizontal 
divisor (cf. Section \ref{s:uncountable-limit-point}).  

\begin{theoremletter}
\label{p:uncountable-limit-point}
Assume that $R$ is infinite. 
Then the following properties hold: 
\begin{enumerate} [\rm (i)]
\item
$I(R, \DD)$ is countably infinite but it is not analytically closed in $B$;  
\item 
the set $I(R, \DD)_\infty$ of limit points of $I(R, \DD)$  in $B$ is uncountable; 
\item  
$I(R, \DD)_\infty$ is not contained  in any union $W \supset  T$ of disjoint closed discs in $B$ such
that distinct points of T are contained in distinct discs. 
 \end{enumerate}
\end{theoremletter}

Here are some motivations for Theorem \ref{p:uncountable-limit-point}.  
Assume that $C$ is a smooth projective curve defined over a number field $k$. 
Let $F= k(C)$ and let $\Phi \colon \mathcal{E} \to C$ be a nonisotrivial Jacobian elliptic surface. 
Assume that $\DD$ is the zero section of $\EE$ and $R= \{nQ \colon n \in \N^* \} \subset \mathcal{E}_F(F)$ for 
some non torsion point $Q \in \mathcal{E}_F(F)$.  
Let $I(R, \DD) \coloneqq \cup_{P \in R} \Phi( \sigma_P(C\left(\bar{k}\right) \cap \DD)) \subset C\left(\bar{k} \right)$ 
where $\sigma_P\in \mathcal{E}(C)$ denotes the induced section of $P \in \mathcal{E}_F(F)$. 
Then by \cite[Notes to chapter 3]{zannier-unlikely-intersection}, $I(R, \DD)$ 
is analytically dense in $C(\C)$. 
\par 
Theorem \ref{p:uncountable-limit-point} thus 
provides some evidence 
that analogous density results on the intersection locus could 
be true in higher dimensional abelian varieties over function fields.
In fact, recent results in \cite{demarco-mavraki}  
imply that $I(R, \DD)$ is even equidistributed in $C(\C)$ with respect to a certain Galois-invariant measure. 
Another supporting result is recently given 
in \cite[Corollary C]{phung-19-uniform-noguchi} in the same context of Theorem 
\ref{p:uncountable-limit-point} but without the hypothesis $\Tr_{K / \C}(A)=0$.

\subsection{Application to the generic emptiness of integral points} 

As the space $B^{(s)}$, $s\geq 1$, of subsets $S\subset B$ 
of cardinality at most $s$ has positive dimension, 
the finiteness of the union $I_s$ 
of $(S, \DD)$-integral sections (Theorem \ref{t:parshin-generic-emptiness-higher-dimension})  implies that 
for a general choice of such $S$, there are very few  $(S, \DD)$-integral points. 
We can show an even stronger property (cf. Section \ref{section:corollary-emptiness-higher-dimension}):  
 
\begin{corollaryletter}
\label{c:parshin-generic-emptiness-higher-dimension} 
Let the notation be as in Theorem 
\ref{t:parshin-generic-emptiness-higher-dimension}. 
Assume $\Tr_{K/\C}(A)=0$ and $\DD$ horizontally strictly nef, i.e., 
$\DD \cdot C >0$ for all curves $C \subset \mathcal{A}$ not contained in a fiber. 
Then for $s \in \N$: 
\begin{enumerate} [\rm (i)]
\item
there exists a finite subset $E \subset B$ such that 
for any $S \subset B \setminus E$ with $\# ( S \cap B_0 ) \leq s$, 
the set of $(S, \DD)$-integral sections is empty. 
We can choose $E$ such that 
\begin{equation}
\label{e:quantitative-bound-generic-empty}
\# E \cap B_0 \leq m s (s+1)^{2 \dim A . \rank \pi_1(B_0)}; 
\end{equation}
\item  
there exists a Zariski proper closed subset $\Delta \subset B^{(s)}$ such that for any $S \subset B$ 
of cardinality $s$ whose image $[S] \in B^{(s)} \setminus \Delta$, 
there are no $(S,\DD)$-integral sections. 
\end{enumerate}
\end{corollaryletter}

The subsets $S \subset B$ satisfying 
Corollary \ref{c:parshin-generic-emptiness-higher-dimension}.(i)  can be taken 
as $S= (B \setminus ( B_0 \cup E)) \cup N = (U \setminus E) \cup N$ where 
$N \subset B_0 \setminus E$ is any finite subset of cardinality at most $s$. 
Since $B_0$ can be taken arbitrarily thin (cf. Theorem \ref{t:parshin-generic-emptiness-higher-dimension}), 
such subsets $S$ are very large. In this sense, we find that 
Corollary \ref{c:parshin-generic-emptiness-higher-dimension}.(i)  is quite surprising.

\subsection{Growth degree of the set of integral  points in Theorem \ref{t:parshin-generic-emptiness-higher-dimension}} 

\label{r:optimal-growth}

 We explain below that the exponent $2 \dim A  \rank \pi_1(B_0)$ in the bound \eqref{e:qualitative-bound-abelian} 
is as reasonable, up to a factor of $1/2$, 
as we can possibly expect. 
In Setting(P), 
let $t= \# T$ and 
let $r$ be the Mordell-Weil rank of $A$. Suppose that  $W$ is a disjoint union of $t$ closed discs centered at the points of $T$ so that rank 
$\pi_1(B_0)=2g-1+t$. 
For $s \in \N$, 
let $J_s \coloneqq \{ P \in A(K) \colon \# f(\sigma_P (B) \cap \DD_z) \leq s \} \subset I_s$. 
\par 
Assume first that $A$ is a nonisotrivial elliptic curve so that its trace is zero. Note
that 
$\frac{1}{2} \deg \mathfrak{f}_{A/K} \leq t \leq \deg \mathfrak{f}_{A/K}$ where $\mathfrak{f}_{A/K}$ is the conductor divisor of 
$A$ over $K$ (cf. \cite[Ex. III.3.36]{silverman-ATAEC}).  
Shioda's result in  \cite[Theorem 2.5]{shioda-elliptic-modular-surface} 
provides a very general bound 
$r \leq2( 2g - 2 + t)$. 
Height theory  \cite[Corollary 8.5]{hindry-silverman-88} 
tells us that 
$\# J_s$ is bounded by a polynomial in $s$ 
of degree 
\begin{equation*}
\frac{r}{2} \leq (2g- 1 + t) =\frac{1}{2} (2 \rank \pi_1(B_0)).
\end{equation*}
\par 
Therefore,  up to a factor of $1/2$, Theorem \ref{t:parshin-generic-emptiness-higher-dimension} 
implies the known polynomial growth for the 
set $J_s \subset I_s$. 
More generally, when $A$ is a traceless abelian variety,  
the Ogg-Shafarevich formula (cf. \cite{oog-shafarevich-formula}, 
\cite{ogg-rank-formula-function-field}, see also \cite{sga-bois-marie}) 
implies  that $r \leq 2 \dim A . (2g - 2 + t)$. 
In this case, best results using height theory (due to Buium \cite{buium-94}, 
see Section  \ref{r:discussion-poly-bound-intro}) 
tell us that $\# J_s$ is bounded by a polynomial in $s$ 
of degree $\frac{r}{2}  \leq \dim A. (2g-2 + t)$. As $\vol_d (W)$ can be
arbitrarily closed to $\vol_d(B)$, our union $I_s$  is \textit{a priori} much larger than
$J_s$. 
However, Theorem 
  \ref{t:parshin-generic-emptiness-higher-dimension} assures 
that the polynomial growth degree of $\# I_s$ 
is still at most $O(\dim A. \rank \pi_1(B_0))$ just as $J_s$. 
\par
Suppose that  $\mathcal{A}$ is a constant family of simple abelian varieties. Then $T=\varnothing$ and  
Noguchi-Winkelmann's results (cf. \cite{noguchi-winkelmann-04}, see Section  \ref{r:discussion-poly-bound-intro}) for a constant ample divisor $\mathcal{D} = D \times B$ 
imply  that 
modulo
 $\Tr_{K/\C}(A)(\C)$, the cardinality of $J_s$ is bounded by a polynomial in $s$ of degree
$\frac{r}{2}
\leq 2g \dim A$  (cf. Remark 1.6 and [35, Theorem J] for more general $\mathcal{D}$). 
Theorem \ref{t:parshin-generic-emptiness-higher-dimension} 
 thus also improves known upper bounds on the growth degree 
to the much larger set $I_s \supset J_s$.

\subsection{On the Geometric Lang-Vojta conjecture} 
\label{r:discussion-poly-bound-intro}

Keep the notation as in Theorem \ref{t:parshin-generic-emptiness-higher-dimension}. 
By a counting argument  
on the lattice 
$\Gamma= {A}(K)/\Tr_{K/\C}(A)(\C)$, 
a bound on the canonical height of $(S, \DD)$-integral sections, 
if it is a polynomial $p $ in $\#S$, 
would imply a polynomial bound 
of the form $(c p(\# S)^{1/2}+1)^{\rank \Gamma}$  
on the number of such integral sections (modulo the trace).  
\par 
Conversely, to obtain with height theory a polynomial bound in function of  $\#S$ on the set of $(S, \DD)$-integral sections with $S \subset B$ finite, 
it is necessary to establish a uniform bound 
on the intersection multiplicities (as predicted by Conjecture~\ref{conjecture-lang-vojta-geometric}). 
Such uniform bounds 
are only available  when  $\Tr_{K/\C}(A)=0$  
 (\cite{buium-94}, 
or \cite{hindry-silverman-88} for elliptic curves) 
or when the family $\mathcal{A} \to B$ is trivial 
(\cite{noguchi-winkelmann-04} for $\DD$ constant, or \cite{phung-19-uniform-noguchi} for general  $\DD$). 

\subsection{Other remarks}
Theorem \ref{t:linear-bound-s-base-curve-1} 
and Theorem \ref{t:linear-bound-sqrt-s-log-s-base-curve-1} 
are in a sense \emph{orthogonal} to various remarkable  
results on the growth of the   
counting functions $c_X(L), s_X(L)$ 
for the number of certain type 
of closed geodesics of length $\leq L$ on 
a    complete hyperbolic bordered Riemann surface $X$  of finite area  
(cf., for example, \cite{margulis-growth}, \cite{birman-series-growth}, 
\cite{irvin-growth}, \cite{mirzakhani-growth}).
\par 
Let $X$ be a complete hyperbolic Riemann surface of genus $g$ with $n$ cusps. 
A \emph{partition} on $X$ is a set of $3g - 3+n$ pairwise disjoint simple closed geodesics. These curves \textit{do not} 
generate the free homotopy group $\pi_1(X)$. 
Then Bers' theorem (cf.~\cite[Theorem 5.2.6]{buser-book-92}) asserts that $X$ admits a
partition 
with geodesics of hyperbolic length 
bounded by $13(3g-3+n)$. 
Hence, Bers' theorem applies to  surfaces of \emph{finite area}. On the other hand, 
punctured  
Riemann surfaces $B_0 \setminus S$ (as in Theorem \ref{t:linear-bound-s-base-curve-1}, 
Theorem \ref{t:linear-bound-sqrt-s-log-s-base-curve-1}) 
have \emph{infinite area} (whenever $U$ is not finite) 
and look like the punctured Poincar\' e 
disc locally around the punctured points. 
Therefore,  Bers' theorems do not apply to the punctured  
Riemann surfaces $B_0 \setminus S$ which are  
equipped with the intrinsic hyperbolic metric $d_{B_0 \setminus S}$ 
(cf. Definition \ref{pseudo Kobayashi hyperbolic metric}). 
Moreover, our  proofs of Theorem \ref{t:linear-bound-s-base-curve-1} and of 
Theorem \ref{t:linear-bound-sqrt-s-log-s-base-curve-1} 
  work with the very definition of the pseudo Kobayashi hyperbolic metric  
 and do not require any 
tools from hyperbolic trigonometry.


\section{Preliminaries for Theorem \ref{t:parshin-generic-emptiness-higher-dimension} and Theorem \ref{t:linear-bound-s-base-curve-1}}
\label{A-weak-form-of-Theorem-B} 
  
Here are some technical difficulties that we must tackle carefully in the  proof of  Theorem~\ref{t:linear-bound-s-base-curve-1}:

\begin{enumerate} [\rm (1)] 
\item 
when $S$ \emph{varies}, 
the hyperbolic metric on $B_0 \setminus S$ 
is not the same nor induced 
by a \emph{single} given metric on $B_0$; 
  note that the analysis of the hyperbolic metric on the punctured complex plan 
$\C \setminus \{a_1, \dots, a_s\}$ ($s \geq 2$) 
is   a nontrivial research area  (cf., e.g.   
\cite{minda-85}, \cite{brooks-platonic}, 
\cite{sugawa-vuorinen});   
 
\item
the base point $b_0$ should not be fixed 
since otherwise, an accumulation of many points 
of the set $S$ near $b_0$  
would increase to infinity the hyperbolic length of loops based at $b_0$; 
\item
a certain construction on loops need to be carried out to 
obtain a bound which is \emph{linear} in $\# S$ but independent 
of the choice of $S$ in $B_0$. 
One should consider only certain classes of "nice" loops and avoid 
pathological loops such as Peano curves. 
\end{enumerate}

To deal with the last point, we shall describe 
a detailed \emph{algorithm} on \emph{simple} loops 
called   Procedure $(*)$ given 
in Lemma \ref{l:linear-bound-hyperbolic-length} (see also 
Lemma \ref{l:base-change-loops} for the global case).

\subsection{Simple base of loops} 
\label{l:primitive-simple-basis-fundamental-group}
Let $X$ be an orientable connected compact surface possibly with boundary. A loop $\gamma \colon [0,1] \to X$  
 is  called \emph{simple} 
if it is non-nullhomotopic and injective on $[0, 1[$. 
Then $\pi_1(X, x_0)$, for every $x_0 \in X$, 
admits a canonical system of  generators $\alpha_1, \dots ,\alpha_k$ 
such that each $\alpha_i$ is represented by a simple piecewise smooth closed loop $\gamma_i \colon [0, 1] \to X \setminus \partial X$ 
based at $x_0$.  
Every such system of generators $\alpha_1, \dots ,\alpha_k$ 
 is called a \emph{simple base} 
of $\pi_1(X, x_0)$.   
\par 

\begin{remark} 
\label{remark-simple-piece-wise-smooth} 
Given a simple piecewise smooth loop $\gamma$ 
in a compact Riemannian surface $(X,d)$ 
with $\gamma \cap \partial X= \varnothing$. 
Since $\gamma$ only
has a finite number of singular points, for every $x \in \gamma$ and $\varepsilon>0$ smaller than the injectivity radius of $X$ at $x$, 
we can find a  contractible closed region $\Delta \subset V(x, \varepsilon)$ 
where $V(x, \varepsilon)$ is the closed disc of $d$-radius $\varepsilon$,  
such that $x \in \Delta$ and 
\begin{enumerate}[\rm (a)] 
\item
$\partial \Delta$ is a non self-intersecting smooth loop homeomorphic to a circle in $V(x, \varepsilon)$;  
\item
$\gamma \cap \Delta$ contains only one connected branch of $\gamma$. 
\end{enumerate} 
\end{remark}


\begin{definition}
\label{d:conjugation-path-1}
Let $B$ be a compact Riemannian surface equipped with a  Riemannian metric $d$. Let $U$ be
finite union of disjoint closed discs in $B$ and $b_0 \in B_0\coloneqq  B \setminus U$. Fix a simple base 
$\alpha_1, \dots , \alpha_k$ of $\pi_1(B_0, b_0)$ and  a collection $\{c_{b_0b}\}$ consisting of bounded $d$-length and smooth directed paths
contained in $B_0$ such that $c_{b_0b}$ goes from  $b_0$ to $b \in B_0$ for each $b \in B_0$. 
A loop $\gamma_i$ represents a class $\alpha_i \in  \pi_1(B_0, b_0)$ \textit{up to a
single conjugation} (with respect to the fixed collection of paths $\{c_{b_0b}\}$) 
if $\gamma_i$ represents the conjugation of the class $\alpha_i \in  \pi_1(B_0, b_0)$ by the change of base points
from $b$ to $b_0$ using the specific chosen path $c_{b_0b}$, i.e., $[c_{b_0b}^{-1}
\circ  \gamma_i \circ c_{b_0b}] = \alpha_i \in  \pi_1(B_0, b_0)$. 
\end{definition}

\subsection{Preliminary lemmata}
\label{chap-linear-bound-hyper-preliminaries}
We denote  by 
$\Delta(x, r) \subset \C$ the open complex disc 
centered at a point $x \in \C$ and of radius $r> 0$. 
For a complex space $X$, 
the infinitesimal Kobayashi-Royden pseudo metric $\lambda_X$ on $X$ 
corresponding to the  
Kobayashi pseudo hyperbolic metric $d_X$  
can be defined as follows. 
For $x \in X$ and every vector $v$ in the tangent cone of $X$ at $x$, we define 
\begin{equation}
\lambda_X(x, v)\coloneqq \inf \frac{2}{R} 
\end{equation}
where the minimum is taken over all $R >0$ for which there exists a holomorphic map 
$f \colon \Delta(0,R) \to X$ such that $f'(0)=v$.  
Note that when $x \in X$ is regular, the tangent cone of $X$ at $x$ 
is the same as the tangent space $T_x X$. 
We begin with the following simple estimation: 
fix a Riemannian metric $d$ on a compact Riemann surface $B$. 
For every $z \in B$ and $r>0$, let  
$D(z,r) \coloneqq \{ b \in B \colon d(b,z) <r \}$. Define the \emph{$d$-unit tangent space} of $B$ by $ T_1 B \coloneqq \{(z, v) \in T B \colon |v|_d=1\}$. By a direct  argument using the  compactness of $B$, we have: 
 
\begin{lemma}
\label{l:fundamtental-estimate-hyperbolic-surface}
There exist  $c(B,d), r(B,d)>0$ 
 such that 
for every $(z, v) \in T_1 B$ and every $0 < r < r(B,d)$, 
we have 
\pushQED{\qed}
$\lambda_{D(z,r)}(z,v) \leq \frac{c(B,d)}{r}$. 
\qedhere
\popQED
\end{lemma} 
\par 

Let the constants $c(B,d)$, $r(B,d)>0$ be as in Lemma 
\ref{l:fundamtental-estimate-hyperbolic-surface} above. 
For every subset $\Omega \subset B$ and  every $r >0$, 
we define $ 
D(\Omega,r) \coloneqq \{b \in B \colon d(b, \Omega) < r \} \subset B$. 
Using the distance-decreasing property of the Kobayashi pseudo hyperbolic metric 
(cf. Lemma \ref{l:distance-decreasing-hyperbolic}), 
we obtain: 
 
\begin{corollary}
\label{c:bound-length-main-hyper-loop-corollary}
Let $\gamma \subset B$ be a piecewise smooth closed curve. 
Then for $0< r< r (B,d)$: 
\[
\pushQED{\qed}
\length_{d_{D(\gamma, r)}}(\gamma)\leq \frac{c(B,d)}{r} \length_d(\gamma).
\qedhere 
\popQED
\]
\end{corollary}

\begin{lemma}
\label{l:length-boundary}
Let $(M, d)$ be a compact Riemannian surface possibly with boundary.
Then there exist constants $r_0=r_0(M,d)>0$, $c_0=c_0(M,d) > 0$ such that for every disc $D(x, r)$ of 
$d$-radius $r \leq r_0$ with $x \in M$, one has $
\length_{d} \partial{D} (x,r)  \leq c_0r$ and   $\vol_d (D(x,r)) \leq c_0 r^2.$
\end{lemma}

\begin{proof}
It follows from the compactness of $M$ and  Bertrand-Diguet-Puiseux's theorem \cite{spivak-bertrand-Diguet-Puiseux}. 
\end{proof}

\subsection{Procedure $(*)$ for simple loops}  
\label{procedure-star}
Given a subset $\Omega$ of a metric space $(X,d)$ and let $r \geq 0$, we recall the notation $
V(\Omega, r) =  \{ x \in X \colon d(x, \Omega) \leq r \}$ and $ D(\Omega, r) =  \{ x \in X \colon d(x, \Omega) < r \}$. 
Let $(B,d)$ be a compact Riemannian surface with  
boundary. 
Let $\gamma \subset B \setminus \partial B$ be \emph{simple} piecewise smooth loop 
based at $b_0 \in B \setminus \partial B$.   
Denote by $\mathrm{rad}(\gamma,d)$ the infimum 
of the injectivity radii in $(B \setminus \partial B,d)$ at all points $x \in \gamma$. 
Since $\gamma$ is compact, 
$\mathrm{rad}(\gamma,d)>0$ and 
$$
L= L(B, d, \gamma) \coloneqq \min ( d(\gamma, \partial B), \mathrm{rad}(\gamma,d))>0. 
$$
\par
We shall consider $a>0$ small enough (depending only on $\gamma$, $B$, $d$) 
such that:  

\begin{enumerate}[\rm (P1)]
\label{condition-on-a-small-star} 
\item
$4a < L(B,d, \gamma)$.  
\end{enumerate}
 
\begin{enumerate} [\rm (P2)]
\item
every point $x \in \gamma$ admits a simply connected open neighborhood $U_x$ in $B\setminus \partial B$ 
such that 
$V(x, 2a) \subset U_x$    
and that   $\gamma  \cap U_x$ 
contains exactly one connected branch of $\gamma$.  
 \end{enumerate} 

By a direct compactness argument, it is not hard to see that:  
 
\begin{lemma}
\label{l:exist-a-gamma-b-d}
There exists $a >0$   
depending only on $\gamma$, $B$, $d$ which satisfies $\mathrm{(P1)}$ and $\mathrm{(P2)}$. \qed
\end{lemma}

\begin{lemma}
[Procedure $(*)$] 
\label{l:linear-bound-hyperbolic-length}
Suppose that $a>0$ satisfies $\mathrm{(P1)}$ and $\mathrm{(P2)}$. 
Then for every finite subset $S \subset B$  
of cardinality $s >0$ such that  
$d(b_0, S)= \min_{x \in S} d(b,x) > a/s$,   
there exists a piecewise smooth loop $\gamma' \subset B \setminus D(S \cup \partial B, a/s)$ based at $b_0$ 
of the same homotopy class in $\pi_1(B, b_0)$ as $\gamma$ and such that  $ 
\length_{d} (\gamma') \leq \length_d(\gamma) +c_0(B,d)a$.  
\end{lemma}
     \begin{figure}[ht]
  \centering
  \includegraphics[page=1,height=.3\textwidth, width=.65\textwidth]{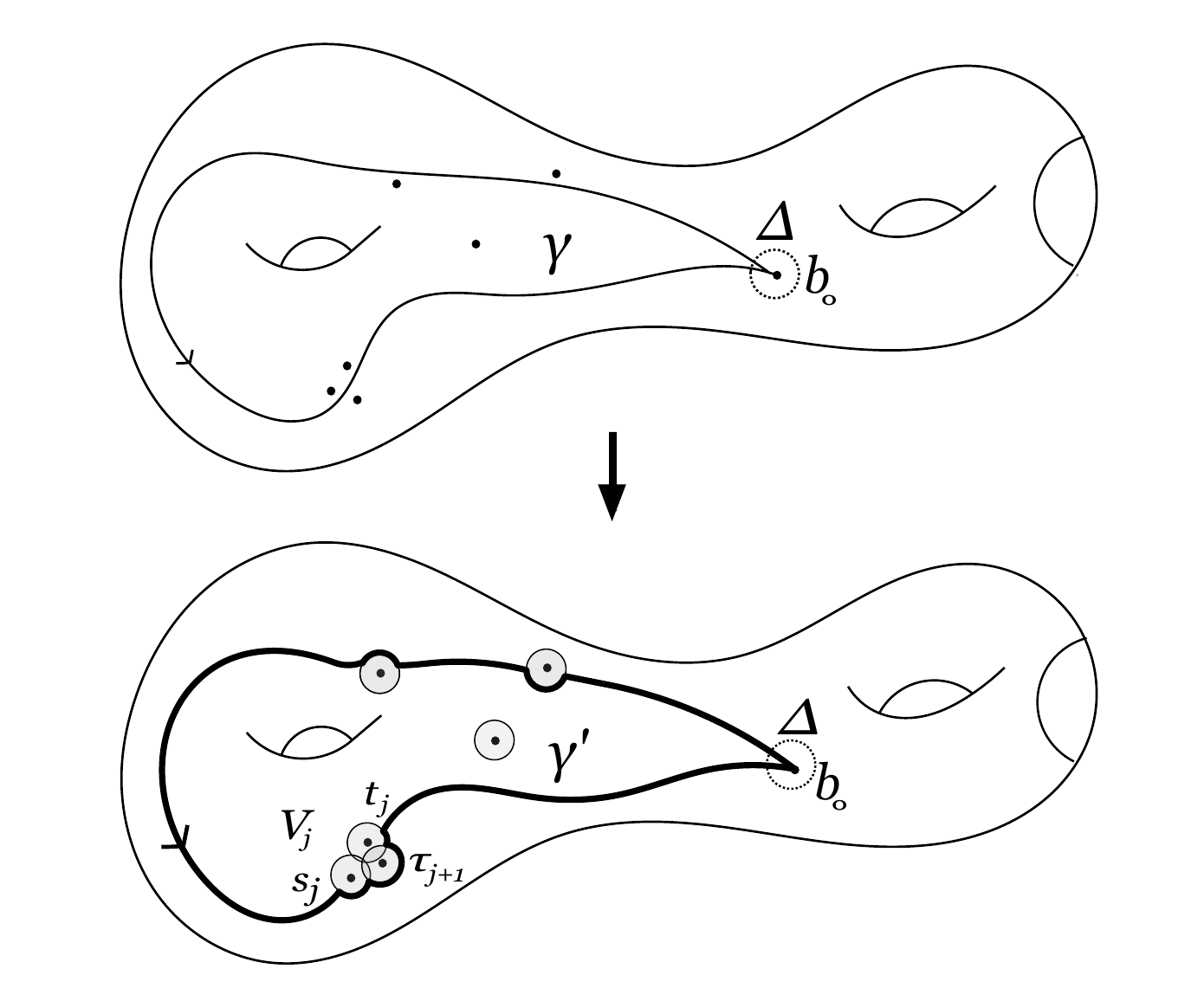}\hspace*{.0\textwidth}%
\caption{Procedure $(*)$ applied to $\gamma$ (with $\Delta= V(b_0,a/s)$)}
    \label{fig:riemannsurface}
\end{figure}

\begin{proof}
 We decompose $V(S, a/s) = \cup_{x \in S} V(x, a/s)$ 
as a disjoint union of connected components $V_1, \dots, V_m$. 
Then each $V_j$ is path connected. 
Let $n_j = \# S \cap V_j$, then:  
\begin{equation} 
\label{e:n-j-star} 
 n_1 + \dots + n_m \leq s= \#S . 
\end{equation} 
\par 
By the triangle inequality, 
$\mathrm{diam}_d (V_j) \leq n_j 2a/s \leq s\times2a/s=2a$. 
 Consider the following $m$-step algorithm. 
Define $\gamma_0 \coloneqq \gamma$. 
Given a curve $\gamma_j \subset B$ where $j \in \{0, \dots, m-1\}$. 
Denote by $s_j, t_j \in \gamma_j$    
the first and the last points, if they exist, 
on the intersection $\gamma_j \cap V_{j+1}$. 
If there are no such points, we set $\gamma_{j+1} = \gamma_j$ and continue the algorithm.   
 Otherwise,  
we replace the directed part  of $\gamma_j$ between $s_j$ and $t_{j}$ by any directed simple curve $\tau_{j+1}$ 
lying on the boundary of $V_j$ which connects $s_j$ and $t_j$ 
(cf. Figure \ref{fig:riemannsurface}). 
This is possible since $V_j$ is path connected. 
Define $\gamma_{j+1}$ the resulting curve and continue until we reach $\gamma_m$. 
 \par
As $ \min_{x \in S} d(b,x) > a/s$ by hypothesis, $b \notin  {V_j}$ for every $j$ 
and thus 
 the base point $b_0 \in \gamma$ is not modified at any step.  
The loops $\gamma_1, \dots, \gamma_m$ based at $b_0$ are piecewise smooth.  
 Moreover,  
 \begin{equation} 
\label{e:n-j-star-V-j} 
\length_{d} (\gamma_{j+1}) 
\leq \length_d (\gamma_{j}) + \length_d (\tau_{j+1}), \quad \quad j =0,\dots,m-1. 
\end{equation} 
\par 
As 
$V_1, \dots, V_m$ are disjoint,  
an induction shows that $s_j, t_j \in \gamma \cap \gamma_j$ and 
each of the curves $\tau_1, \dots, \tau_{j}$  (when defined) is either non modified or does not appear in 
$\gamma_{j+1}$ at the $j$-th step for every 
$0 \leq j \leq m-1$. 
Hence, $\gamma_m$ contains some  pairwise disjoint curves 
$\tau_{i_1}, \dots, \tau_{i_k}$ and $\gamma_m$ 
coincides with $\gamma$ outside 
the directed paths $\sigma_{i_p} \subset \gamma$ 
between $s_{i_p-1}$ and $t_{i_p-1}$ where $p=1, \dots, k$. 
It follows that 
there exists a homotopy in $B$ with base point $b_0$ 
between $\gamma_m$ and~$\gamma$. 
 \par
 Since $\mathrm{diam} V(S, a/s) \leq s \times 2a/s=2a$ and 
$4a < d(\gamma, \partial B)$ by (P1), 
one has 
$d(\gamma_m, \partial B) > 2a$. 
Thus $d(\gamma_m, S \cup \partial B) \geq a/s$ as $s \geq 1$. 
Hence, $\gamma' = \gamma_m$ is a piecewise smooth loop based at $b_0$ 
homotopic to $\gamma$ and $\gamma' \subset B \setminus D(S \cup \partial B, a/s)$. 
Let $c_0=c_0(B,d)>0$ be as in Lemma \ref{l:length-boundary} then:   
\begin{equation}
\label{e:tau-j-star}
\length_d (\tau_j) \leq \sum_{x \in S\cap V_j} \length_d \partial V(x, a/s) \leq n_j \times c_0 a /s. 
\end{equation}
  
It follows from \eqref{e:n-j-star-V-j}, \eqref{e:n-j-star-V-j}, and \eqref{e:n-j-star} that 
 \begin{align*}
\length_{d} (\gamma') 
 \leq \length_d (\gamma) + \sum_j \length_d (\tau_j) 
 \leq \length_d (\gamma) + \sum_j n_j \times c_0 a /s  \leq \length_d( \gamma) + c_0 a. 
\end{align*}
\par 
The proof of the lemma is complete. 
 \end{proof}

\par 
Let $U \subset B\setminus \partial B$ be a simply connected open neighborhood 
of $b_0$ 
such that $\gamma  \cap U$ 
contains exactly one connected branch of $\gamma$. 
Let $\Delta \subset B \setminus \partial B$ be a closed subset such that:  
\begin{enumerate} [\rm (a)]
\item 
$b_0 \subset \Delta \subset U$; 
\item
for every $x \in \Delta$, there exists a simple loop $\gamma_x \subset B$ 
based at $x$ such that $\gamma_x$ and $\gamma$ coincide as paths 
over $B \setminus \Delta$. 
\end{enumerate}
 
\begin{lemma}
\label{l:exist-a-gamma-b-d-family}
There exists $A>0$ such that for every $x \in \Delta$, 
the constant $A$ satisfies $\mathrm{(P1)}$ and $\mathrm{(P2)}$ for 
the loop $\gamma_x$ in the Riemannian surface $(B,d)$. 
\end{lemma}

\begin{proof}
As $\Gamma = \gamma \cup \Delta \subset B \setminus \partial B$ is a compact subset,  $L \coloneqq \min ( d(\Gamma, \partial B), \mathrm{rad}(\Gamma,d))>0$ where $\mathrm{rad}(\Gamma, d)$ is the infimum  
of the injectivity radii in $(B \setminus \partial B,d)$ at all points $x \in \Gamma$. 
For every $x \in \Delta$, we have $\gamma_x \subset \Gamma$ by Condition (b) above and 
thus $L(B,d, \gamma_x) \geq L$. 
Suppose on the contrary that for every $n \geq 5$, there exist 
$x_n \in \Delta$ and $z_n \in \gamma_x$ 
such that for all simply connected open neighborhood $U_n$ of $x_n$ in $B$ with  
$V(z_n, 2L/n) \subset U_n$, 
the restriction $\gamma_x \cap U_n$ has more than two connected components. 
By the compactness of $\Gamma$, we can suppose, up to passing to a subsequence, 
that $z_n \to z$ as $n \to \infty$ for some $z \in \Gamma$. 
We distinguish two cases. 
First, assume that $z \in U$. 
Then for $n \gg 1$, we have 
$z_n \in U$ and $V(z_n, 2L/n) \subset U$. 
This is a contradiction since $U \subset B \setminus \partial B$ 
is simply connected by hypothesis and 
$\gamma_x \cap U$ has only one connected branch of $\gamma_x$ by Condition (b). 
If $z \in B \setminus U$ then 
we also find a contradiction 
since $\gamma_x$ and $\gamma$ coincide as paths 
over $B \setminus \Delta$ and $\Delta$ is closed and contained in $U$. 
\end{proof}

\section[Bound of length of loops with 
varying base points]{Riemannian lengths of special loops with 
varying base points}
\label{Bound of length of loops with varying base points}

We describe a global version of Procedure $(*)$ 
given in Lemma \ref{l:linear-bound-hyperbolic-length}.
Let $(B,d)$ be a compact Riemannianian surface. 
Let $V$ be a finite union of disjoint closed discs in $B$. 
Let $U= V \setminus \partial V$. 
Let  $b_0 \in B_0 \coloneqq B \setminus U$ 
and fix a base of simple generators $\alpha_1, \dots, \alpha_k$ 
of $\pi_1(B_0, b_0)$ (Section ~\ref{l:primitive-simple-basis-fundamental-group}). 
 
\begin{lemma}
\label{l:base-change-loops} 
Let $\varepsilon >0$. Then there exist constants $a, H >0$ with the following property. 
For every $x \in B_0$ such that $d(x, \partial B_0) \geq \varepsilon$, 
there exist simple piecewise smooth 
 loops $\gamma_1, \dots, \gamma_k \subset B_0$ based at $x$ representing $\alpha_1, \dots, \alpha_k$ up to a single conjugation (Definition~ \ref{d:conjugation-path-1}) such that $a$ verifies $\mathrm{(P1)}$ and $\mathrm{(P2)}$ for the data $(\gamma_i, B_0, d)$ and 
$ \length_d (\gamma_i) \leq H$  \text{for every } $i \in \{1, \dots, k\}$.  
\end{lemma}

\begin{proof}
For each $b \in B_0$ with $d(b, \partial B_0) \geq \varepsilon$, we can choose 
 simple piecewise smooth loops 
$\gamma_{b, 1}, \cdots, \gamma_{b,k} \subset B_0 \setminus \partial B_0$ 
based at $b$  representing the classes $\alpha_1, \dots, \alpha_k$ respectively up to 
a single conjugation. 
Let  $ 
H_b  \coloneqq \max_{1 \leq i \leq k} ( \length_d \gamma_{b, i} ) >0$. 
\par 
By Lemma \ref{l:exist-a-gamma-b-d}, 
there exists $a_b >0$ satisfying  (P1) and (P2) for ($\gamma_{b, i}$, $B_0$, $d$) for all $i$. 
Thus, for each $i\in \{1, \dots, k\}$, 
$b$ admits a simply connected neighborhood $U_{b,i} \subset B_0 \setminus \partial B_0$ 
such that $V(b, 2a_b) \subset U_{b,i}$ and  $\gamma  \cap U_{b,i}$ 
contains only one connected branch of $\gamma_{b,i}$ (cf. Remark \ref{remark-simple-piece-wise-smooth}). 
\par
Consider a small enough closed region $\Delta_b \subset V(b, a_b) \subset B_0 \setminus \partial B_0 $ containing $b$ 
such that $\partial \Delta_b$ is a non self-intersecting smooth loop homeomorphic to a circle in $V(b, a_b)$  
and $\gamma_{b, i} \cap \Delta_b$ contains only one connected branch of $\gamma$ for every $i=1, \dots, k$. 
Let $l_b\coloneqq \length_d(\partial \Delta_b)>0 $. 
 \par
Since $B'= \{b \in B_0 \colon d(x, \partial B_0) \geq \varepsilon \}$ is compact, 
there exists a finite subset $I \subset B'$ such that 
$B' \subset  \cup_{b \in I} \Delta_b$. 
 \begin{figure}[ht]
   \centering
  \includegraphics[page=1,height=.25\textwidth, width=.65\textwidth]{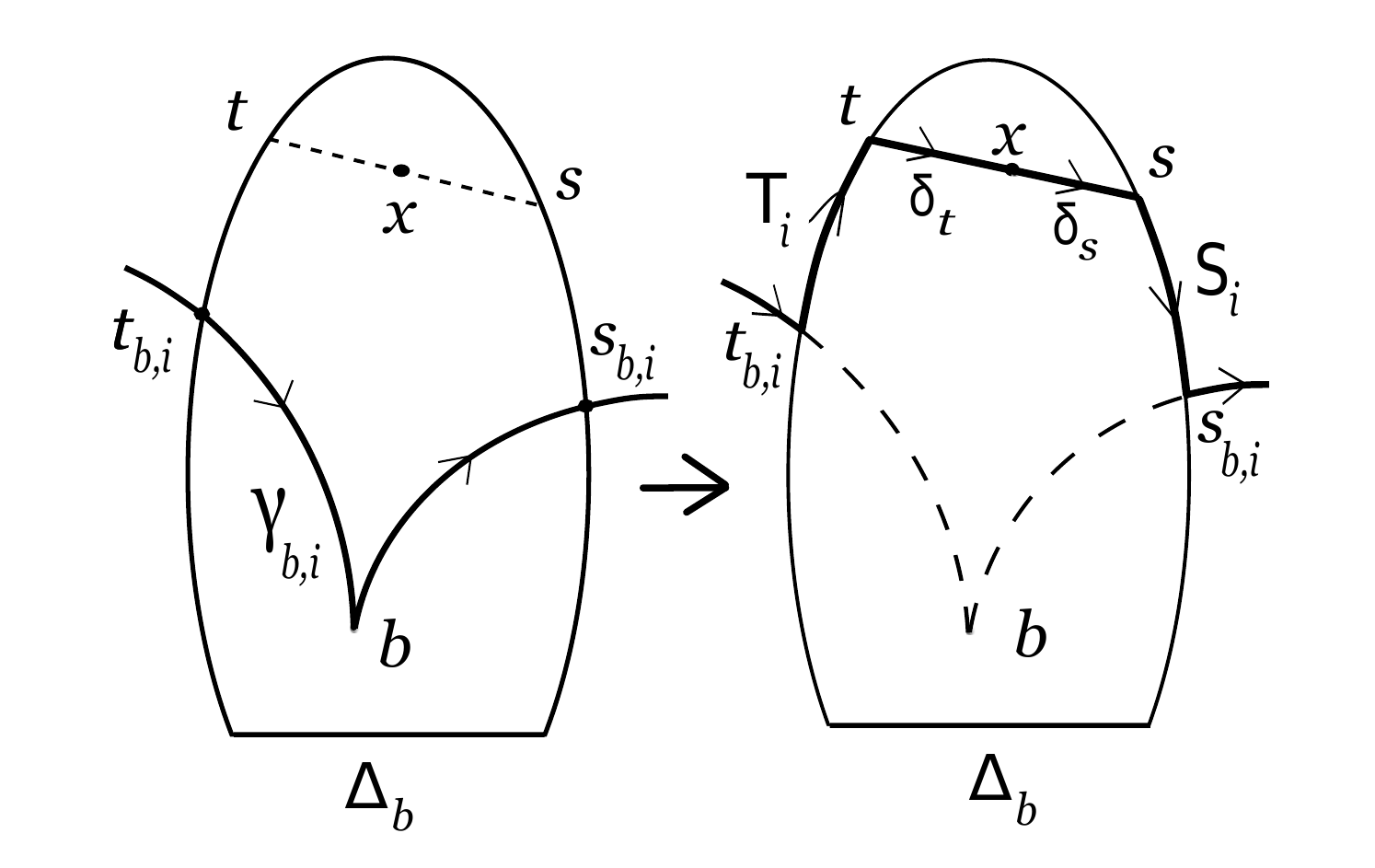}\hspace*{.0\textwidth}%
\caption{Local modification of $\gamma_{b,i} \cap {\Delta_b}$}
  \label{fig:lemma-1}
\end{figure}
Consider the following construction for every $x \in B'$ . 
We can choose $b \in I$ such that $x \in \Delta_b$. 
 For each $i \in \{1, \dots, k\}$, 
let $s_{b,i},  t_{b,i} \in \gamma_{b,i} \cap \Delta_b$ be respectively the first and the last 
intersections of $\gamma_{b,i}$ with the boundary $\partial \Delta_b$. 
 Note that $s_{b,i} \neq  t_{b,i}$. 
Let $\sigma \subset  \Delta_b$ be any maximal geodesic segment passing through $x$. 
Let $s, t \in \sigma \cap \partial \Delta_b$ be the two extremal points of $\sigma$ such 
that $s$ and $s_{b,i}$ do not lie on distinct arcs delimited by $t_{b_i}$ 
and $t$ on $\partial{\Delta_b}$ (cf. Figure \ref{fig:lemma-1}). 
The two directed geodesic segments $\delta_s$ from $x$ to $s$ and 
$\delta_t$ from $t$ to $x$ do not intersect except at $x$. 
Replace $\gamma_{b,i} \cap \Delta_b$ by 
the union of the directed arc $T_i \subset \partial \Delta_b$ from $t_{b,i}$ 
to $t$ not containing $s_{b,i}$, with the paths $\delta_t$, 
$\delta_s$ and the directed arc $S_i \subset  \partial \Delta_b$ from $s$  
to $s_{b,i}$ not containing $t_{b,i}$. 
 Denote the resulting loop 
with base point $x$ by $\gamma_{x,i}$. 
\par
By construction, the loop $\gamma_{x,i}$ 
is simple and piecewise smooth. 
Moreover, the restrictions of $\gamma_{b,i}$ and $\gamma_{x,i}$ to $B_0 \setminus \Delta_b$ coincide. 
As $\Delta_b \subset U_{b,i}$ and $U_{b,i}$ is simply connected and contains 
only one branch of $\gamma_{b,i}$, the loops $\gamma_{b,i}$ and $\gamma_{x,i}$ 
are homotopic up to a conjugation induced by the change of base points.  
By setting 
$H \coloneqq  \max_{b \in I} ( H_b +l_b+2a_b )>0$, 
we find that:   
\begin{align*}
\length_d (\gamma_{x,i}) 
& = \length_d(\gamma_{x,i}\vert_{B \setminus \Delta_b})  \\
& \quad + \length_d (T_i) 
+ \length_d(S_i)  
+ \length_d(\delta_t) + \length_d( \delta_s) 
\\ 
& \leq \length_d(\gamma_{b,i}) + \length_d(\partial \Delta_b) 
+ \length_d(\sigma) 
\\
& \leq  H_b +  l_b + 2a_b \leq H.  \quad \quad \quad  (\text{As } \sigma \subset \Delta_b \subset V(b, a_b))
\end{align*}
\par 
By construction, $\Delta_b \subset V(b, a_b) \subset U_{b,i}$ for 
every  $b \in I$ and $i \in \{1, \dots, k\}$. 
Therefore, 
Lemma~\ref{l:exist-a-gamma-b-d-family} applied to $\gamma_{b,i}$, $b$, 
$U_b$, and $\Delta_b$  
implies that   
there exists  $A_{b,i}>0$ 
such that for every $x \in \Delta_b$, 
the constant $A_{b,i}$ satisfies Properties 
$\mathrm{(P1)}$ and $\mathrm{(P2)}$ for the loop $\gamma_{x,i}$ in the Riemannian 
surface $(B_0, d)$. 
As $I$ is finite, we can define $
a \coloneqq \min_{b \in I} \min_{1 \leq i \leq k} A_{b,i} >0$.  
\par 
It is clear that $H,a>0$ are independent of $x \in B'$ and verify the 
desired properties.  
\end{proof}

\section{Proof of Theorem \ref{t:linear-bound-s-base-curve-1}}  
\label{proof-of-hyper-linear-bound-chapter}
Let $U$ be a finite union of disjoint closed discs in a compact Riemann surface $B$. 
Let $b_1 \in B_1 \coloneqq B \setminus U$. 
Fix a simple base  $\alpha_1 , \dots, \alpha_k$ of $\pi_1(B_1, b_1)$ (see Section~ \ref{l:primitive-simple-basis-fundamental-group}). 
Let $d$ be a Riemannian metric on $B$. 
We prove the following slightly stronger version of  
Theorem  \ref{t:linear-bound-s-base-curve-1}. 

\begin{theorem}
\label{t:collars-for-linear-hyperbolic-bound}
Let the hypotheses be as above. 
Then there exist $A>0$ and $L >0$ satisfying the following properties. 
For any finite subset $S \subset B_1$,  
there exist $b \in B_1\setminus S$ and piecewise smooth loops 
$\gamma_i \subset B_1 \setminus S$ based at $b$ representing   
$\alpha_i$ up to a single  conjugation and such that 
 $V(\gamma_i, A( \# S)^{-1}) \subset B_1\setminus S$ and $
\length_{d_{B_1 \setminus S}} (\gamma_i) \leq L( \#S + 1)$ for  $i=1, \dots, k$.
\end{theorem}

\begin{proof}
Fix $\varepsilon >0$ small enough so that 
$\vol_d(B_\varepsilon ) >0$ where 
$B_\varepsilon = \{b \in B_0 \colon d(x, \partial B_0) \geq \varepsilon \}$. 
Let $a, H>0$ be the constants given by Lemma \ref{l:base-change-loops}  
applied to $(B_0, d)$,  
to the constant  $\varepsilon>0$ and  to the base 
$\alpha_1, \dots, \alpha_k$ 
of $\pi_1(B_0, b_1)= \pi_1(B_1, b_1)$. 
Let $c_0=c_0(B_0, d)$,  $r_0=r_0(B_0,d)$, $c=c(B,d)>0$,  and $r=r(B,d)>0$ 
be the constants given by Lemma \ref{l:length-boundary} and Lemma \ref{l:fundamtental-estimate-hyperbolic-surface} applied 
to $(B_0,d)$ and $(B,d)$ respectively.  Let us define: 
\begin{equation}
\label{e:new-a-main-theorem-hyper-length} 
A \coloneqq   \frac{1}{2}\min \left(a, r, r_0 , \sqrt{\frac{ \vol_d (B_\varepsilon)}{4c_0} } \right) > 0, 
\quad \quad 
L \coloneqq c \left( \frac{H}{A}+ c_0 \right) >0. 
\end{equation}
 \par 
Let $S \subset B_0$ be a finite subset of cardinality $s \geq 1$.  
We find that: 
 \begin{align*}
 \vol_d V(S,2A/s)  
& \leq  \sum_{x \in S}\vol_d( V(x,2A/s) ) & (\text{as } V(S,2A/s) = \cup_{x \in S} V(x,2A/s))
\\
& \leq  s \times c_0 \times (2A/s)^2 
=  4c_0A^2 /s
& (\text{by Lemma } \ref{l:length-boundary} \text{ and } 2A/s \leq r_0 
\text{ by }\eqref{e:new-a-main-theorem-hyper-length})  
\\ 
& \leq 4 c_0 A^2 \leq  \frac{\vol_d(B_\varepsilon)}{4}< \vol_d(B_\varepsilon).   
& (\text{since } s \geq 1 \text{ and by }\eqref{e:new-a-main-theorem-hyper-length}) 
 \end{align*}
 \par 
 It follows that there exists $b \in B_\varepsilon \subset B_1$ such that 
 $b \cap V(S, 2a/s) = \varnothing$, i.e., $d(b, S) \geq 2a/s$. 
Therefore, 
Lemma  \ref{l:base-change-loops} implies that 
there exist simple piecewise smooth 
loops $\sigma_1, \dots, \sigma_k \subset B_1$ based at $b$ representing $\alpha_1, \dots, \alpha_k$ 
respectively up to a single conjugation with: 
\begin{equation}
\label{e:bound-riemann-main-theorem-hyper-length-1}
\length_d (\sigma_i) \leq H, \quad \text{for every } i \in \{1, \dots, k\}. 
\end{equation}
\par 
Moreover, 
as $0<A < a$ by \eqref{e:new-a-main-theorem-hyper-length}, we infer from  
Lemma \ref{l:base-change-loops} that 
$A$ verifies the properties  $\mathrm{(P1)}$ and $\mathrm{(P2)}$ for the data $(\sigma_i, B_0, d)$ 
for every $i \in \{1, \dots, k\}$. 
\par
Since $d(b, S)= \min_{x\in S} d(b,x) \geq 2a/s>A/s$ by \eqref{e:new-a-main-theorem-hyper-length}, 
we can   apply Lemma \ref{l:linear-bound-hyperbolic-length} 
to the loops $\sigma_1, \dots, \sigma_k$ 
to obtain 
piecewise smooth loops 
$\gamma_1, \dots, \gamma_k \subset B_0 \setminus V(S \cup \partial B_0, A/s)$ 
based at $b$ 
which are of the same homotopy classes  as $\sigma_1, \dots, \sigma_k$ in $\pi_1(B_0, b_1)$ respectively 
and satisfy: 
\begin{equation}
\label{e:bound-riemann-main-theorem-hyper-length}
\length_{d} (\gamma_i) \leq \length_d(\sigma_i) + c_0 A, \quad i = 1, \dots, k.   
\end{equation} 
\par 
In particular, $V(\gamma_i, A/s) \subset B_1 \setminus  S$ for every $i$.  Consequently, we find that: 
 \begin{align*} 
\length_{d_{B_1 \setminus S}} (\gamma_i) 
& \leq \length_{d_{V(\gamma_i, A/s)}} (\gamma_i) 
& (\text{by Lemma }\ref{l:distance-decreasing-hyperbolic} \text{ as } D(\gamma_i, A/s) \subset B_1 \setminus  S)
\\
 & \leq \frac{c(B,d)}{A/s}  \length_{d} (\gamma_i)  
 &  (\text{by Corollary }\ref{c:bound-length-main-hyper-loop-corollary} 
 \text{ and  }A/s \leq A < r \text{ by } \eqref{e:new-a-main-theorem-hyper-length}) 
 \\
 & \leq  \frac{cs}{A} (\length_d(\sigma_i) + c_0 A) &  (\text{by }\eqref{e:bound-riemann-main-theorem-hyper-length} ) 
 \\
& \leq \frac{cs}{A}(H +  c_0 A) =Ls  & (\text{by } \eqref{e:bound-riemann-main-theorem-hyper-length-1} \text{ and } \eqref{e:new-a-main-theorem-hyper-length})
\end{align*} 
\par 
The conclusion follows since by construction, 
the loops $\gamma_1, \dots, \gamma_k \subset B_1 \setminus S$  
represent the homotopy classes $\alpha_1, \dots, \alpha_k$ 
respectively up to a single conjugation.  
\end{proof}
\par 
A slight modification of the above proof allows us to show the following
generalization of Theorem~\ref{t:linear-bound-s-base-curve-1} and Theorem~\ref{t:collars-for-linear-hyperbolic-bound} by  allowing certain bounded moving
discs besides $S$: 

\begin{theorem}
 \label{t:collar-moving-discs-general} 
 Let the hypotheses be as in Theorem~\ref{t:collars-for-linear-hyperbolic-bound} and let
$p \geq 1$. Then there exist  $L, R > 0$ with the following property. For every finite subset
$S \subset B_0$ and every union $Z$ of $p$ discs in $B$ each of $d$-radius $R$, there exist $b \in B_0 \setminus (S \cup Z)$
and piecewise smooth loops $\gamma_1, \dots, \gamma_k \subset  B_0 \setminus S$ based at $b$ which represent respectively
$\alpha_1, \dots, \alpha_k$ in $\pi_1(B_0, b_0)$ up to a single conjugation and such that $\length_{d_{B_0 \setminus (S \cup Z)}}(\gamma_i) \leq L(\# S +1)$. 
\end{theorem}

\begin{proof}
We adopt the proof of Theorem \ref{t:collars-for-linear-hyperbolic-bound}. Let
$R \coloneqq A/(4p)$.  The constant $A' 
\coloneqq A/4 < A < a$ still satisfies the
conclusion of Lemma~\ref{l:base-change-loops}. Then we can simply take $R$ and $4L > 0$. 
Indeed, for every $p$ discs $Z$ in $B$ each of $d$-radius $R$, the only
two modifications needed  are the following. 
First, by the same area argument, we can find $b \in B_\varepsilon \subset B_1$ such that $b \cap (V (S, 2a/s) \cup 
Z) = \varnothing$. The second change lies in the use of Lemma~\ref{l:linear-bound-hyperbolic-length} in the last step: we apply the
procedure $(*)$ described in the proof of Lemma~\ref{l:linear-bound-hyperbolic-length} for the decomposition into connected
components of the closed set $V(S, A'/s) \cup Z$  instead of the set $V(S, A'/s)$. 
\end{proof}

\section{Proof of Theorem \ref{t:linear-bound-sqrt-s-log-s-base-curve-1}}
\label{proof-of-hyper-lower-bound-chapter}
Let $\Omega =  \C\Proj^1 \setminus \{0, 1, \infty\}$.  
Denote by $T_1 \C \Proj^1$ the unit tangent space with 
respect to the Fubini-Study metric $d_{FS}$ on $\C\Proj^1$ given by 
$d_{FS}z= |dz|/(1+|z|^2)$ where $z$ is the  affine coordinate  
on $\C \Proj^1$. 
The next result of Ahlfors (\cite[Theorem 1-12]{ahlfors-73}, notably (1-24)) says that 
the hyperbolic metric on $\Omega$ near the cusp $0$ 
behaves as the hyperbolic metric of the punctured unit disc: 

\begin{theorem}
[Ahlfors] 
\label{ahlfors-73-basic-lemma}
There exist $\delta > 0$ and $C >0$ such that 
for every $(z, v) \in T_1 \C\Proj^1$ 
with $z \in \Omega$ and $d_{FS}(z, 0) < \delta$, we have $| \ln \lambda_\Omega (z, v) +  \ln d_{FS}(z, 0) + \ln \ln (d_{FS}(z, 0))^{-1} | < C$. \qed
\end{theorem}


\begin{lemma}
\label{l:packing-sphere-by-discs}
There exists 
$r_0 > 0$ such that for every $s \geq 2$, we can 
cover the Riemann sphere $\C \Proj^1$ by $s$ closed discs of $d_{FS}$-radius $r_0s^{-1/2}$. 
\end{lemma}

\begin{proof} 
Let $\Delta_2= \{z \in \C\colon |z| \leq 2\} \subset \C$. 
Denote $d$ the Euclidean metric on $\C$. 
For each $\varepsilon >0$, let $N(\varepsilon)$ denotes the minimum 
number of closed discs of $d$-radius $\varepsilon$ in $\C$ which can cover $\Delta_2$. 
Then Kershner's theorem (cf. \cite{kershner}) tells us that  
 $\lim_{s \to \infty} s^{-1}N(s^{-1/2}) = 8\pi 3^{1/2}/9$. 
\par
Note that $N(s^{-1/2})$ is an increasing function in $s$. 
It follows that there exists a real number $c > 1$ such that $N(s^{-1/2}) < c s$ for all $s \geq 1$. 
Replacing $s$ by $s/c$, we deduce that 
for every $s \geq c$, there exists a covering of $\Delta_2$ by 
at most $ s$ discs of $d$-radius $(c/s)^{1/2}$. 
In particular, since $(c/s)^{1/2} \leq 1$ for $s \geq c$, 
we can find a subset of $k \leq s$ discs $D_1, \dots, D_k$ which cover $\Delta_1$ 
and whose centers $z_1, \dots, z_k$ belong to $\Delta_2$. 
Let  $p_N$ be the stereographic projection from the north pole of $\C \Proj^1$ onto $\C$. 
We obtain a cover of the Southern hemisphere of $\C \Proj^1$ 
by $p_N^{-1}(D_1), \dots, p_N^{-1}(D_k)$.  
As $d_{FS}z= |dz|/(1+|z|^2) \geq |dz|/5$ for every 
$z \in \Delta_2$, every set  $p_N^{-1}D_i$, 
where $i=1, \dots, k$, 
is contained in the disc centered at $p_N^{-1}(z_i)$ 
of $d_{FS}$-radius   $5(c/s)^{1/2}$. 
By symmetry, we obtain   a cover of $\C \Proj^1$ 
by $2s$ discs of $d_{FS}$-radius $\leq 5(c/s)^{1/2}$ for every $s \geq c$.
\end{proof}


\begin{proof}[Proof of Theorem \ref{t:linear-bound-sqrt-s-log-s-base-curve-1}] 
Consider an arbitrary ramified cover  $\pi \colon B \to \C\Proj^1$ 
of $B$ to the Riemann sphere. 
Let $d_{FS}$ be the Fubini-Study metric on $\C\Proj^1$. 
We denote by $\tilde{d}$ the induced metric on $B \setminus R_\pi$ 
where $R_\pi \subset B$ is the branch locus of $\pi$ which is a finite subset. 
\par
Since $\alpha \in \pi_1(B_0) \setminus \{0\}$, 
it is well-known that there exists $c_0>0$ such that every loop $\gamma \subset B_0$ 
representing $\alpha$ has bounded $\tilde{d}$-length from below by $c_0$, i.e., 
$\length_{\tilde{d}} (\gamma) >c_0$ (cf., for example, 
\cite[Theorem 1.6.11]{buser-book-92}).  
In particular, it follows that 
\begin{equation}
\label{e:lower-bound-s-log-s-Euclidean} 
\length_{d_{FS}} (\pi( \gamma)) = \deg(\pi)^{-1}  \length_{\tilde{d}} (\gamma)  > c_0 \deg(\pi)^{-1}.
\end{equation}
 
 By Lemma \ref{l:packing-sphere-by-discs}, 
there exists 
$r_0 > 0$ such that for every $s \geq 2$, we can 
cover the Riemann sphere $\C \Proj^1$ in a certain regular manner by $s$ discs of $d_{FS}$-radius $r_0s^{-1/2}$. 
Let $Z \subset \C\Proj^1$ be the set containing the centers 
of these discs. 
For each point $z \in Z$, let $z' \in \C \Proj^1$ 
be any point on the equator relative to $z$ as a pole. 
Consider the stereographic projection $P_w$ from the 
opposite pole $w \in \C \Proj^1$ of $z$ to the complex plane 
 such that  $P_w(z)=0$ and $P_w(z')=1$. Then 
$P_w$ is a biholomorphic isometry with respect to the induced Fubini-Study metric $d_{FS}$: 
\begin{equation}
\label{e:isometry-fubini-study} 
\tilde{P}_w \colon  \C \Proj^1 \setminus \{w, z, z'\} \to 
\Omega= \C \Proj^1 \setminus \{ 0,1, \infty\}. 
\end{equation}
 
Define $T= \{w,z,z' \colon z \in Z\}$ and $S = \pi^{-1}Z \subset B$. 
Then $\pi (B_0 \setminus S) \subset \C \Proj^1\setminus T$ and: 
\begin{equation}
\label{e:s-roughly-as-s-log}
\# S \sim 3\deg(\pi) s. 
\end{equation}
\par 
For every $x \in \C \Proj^1 \setminus  T$, 
we can find by construction some 
$z \in Z$ such that $d_{FS}(x, z) < r_0s^{-1/2}$. 
By Theorem \ref{ahlfors-73-basic-lemma} applied to $x \in \C \Proj^1 \setminus \{w, z, z'\} \simeq \Omega$, 
we deduce that for all $s$ large enough so that $r_0s^{-1/2}< \delta$ and 
for every unit vector $v \in (T_1 \C\Proj^1)_x$, we have: 
\begin{align*}
\lambda_{\C \Proj^1\setminus T}(x, v) 
&\geq \lambda_{\C \Proj^1\setminus \{w,z,z'\}}(x, v) &  (\text{by Lemma } \ref{l:distance-decreasing-hyperbolic}) 
\numberthis \label{e:ahlfors-73-basic-lemma-application-sphere}
\\
&\gtrsim d_{FS}(x,z)^{-1} (\ln d_{FS}(x,z)^{-1})^{-1} & (\text{by } \eqref{e:isometry-fubini-study} \text{ and Theorem } \ref{ahlfors-73-basic-lemma})
\\
& \gtrsim s^{1/2} (\ln (s))^{-1}.  & (\text{as } d_{FS}(x, z) < r_0s^{-1/2})
\end{align*}
\par 
Now let $\gamma \subset B_0 \setminus S$ be 
any piecewise smooth loop representing $\alpha$ and is 
parametrized by a map   
$f \colon [0,1] \to B_0 \setminus S$.  
We find that: 
\begin{align*}
 \length_{B_0 \setminus S} (\gamma)  
&  = \int_0^1  \lambda_{B\setminus S} \left( f(t), f'(t)) \right) dt &
\\
& \geq \int_0^1  \lambda_{\C\Proj^1 \setminus T} \left(\pi \circ f(t), (\pi \circ f)'(t) \right) dt   
& (\text{by Lemma } \ref{l:distance-decreasing-hyperbolic})
\\
& \gtrsim  \int_0^1  s^{1/2} (\ln (s))^{-1}  |(\pi \circ f)'(t)|_{d_{FS}}   dt 
& (\text{by } \eqref{e:ahlfors-73-basic-lemma-application-sphere})
\\
& = s^{1/2} (\ln  (s))^{-1}  \length_{d_{FS}}(\pi(\gamma)) &
\\
&  \gtrsim \#S^{1/2} (\ln  (\#S+1))^{-1}.  
& (\text{by }\eqref{e:lower-bound-s-log-s-Euclidean} \text{ and } \eqref{e:s-roughly-as-s-log})
\end{align*}
\end{proof}

\begin{remark}
\label{r:optimal-lower-bound-s-log-s} 
The lower asymptotic polynomial growth $s^{1/2}$ in Theorem \ref{t:linear-bound-sqrt-s-log-s-base-curve-1}
is optimal. 
Indeed, assume that $B= \C \Proj^1$. 
Then in the above proof, we can use 
 directly the bi-liptschitz metric equivalence Lemma~\ref{ahlfors-73-basic-lemma} 
in Theorem \ref{ahlfors-73-basic-lemma} 
instead of the inequality \eqref{e:ahlfors-73-basic-lemma-application-sphere}. 
Therefore, for the choice given in the above proof of certain uniform distribution of the 
$s$ points on $\C \Proj^1$, 
the asymptotic polynomial growth $s^{1/2}$ of the function $L(\alpha, s)$ 
is attained.  
\end{remark}

\section{Parshin's homotopy reduction and metric properties of hyperbolic spaces} 
\label{c:Hyperbolic and homotopy method}

\subsection{The homotopy reduction step  of Parshin} 

In Setting (P), 
let $t_A \in \N$ be the cardinality of  $(A(K)/\Tr_{K/\C}(A)(\C))_{tors}$  (cf.~\cite{lang-neron-theorem-paper}). 
Without loss of generality, we suppose in the rest of the paper that $A[m] \subset A(K)$ 
for some integer $m \geq 2$. 
The proof of Theorem \ref{t:parshin-generic-emptiness-higher-dimension} 
 combines the  \emph{homotopy} reduction step due to Parshin 
(cf. Proposition \ref{p:homotopy-rational-abelian}) 
with the estimation given in Theorem \ref{t:linear-bound-s-base-curve-1} 
on the \emph{hyperbolic} lengths of loops in 
 complements of the Riemann surface $B$.   

\par

The bridges connecting the above two blocks
Theorem \ref{t:parshin-generic-emptiness-higher-dimension} 
are the following: 
the first one is the \emph{Fundamental 
Lemma}  of the geometry of groups, which 
we formulated in Proposition \ref{p:same-growth-geometric-group} 
and in Lemma \ref{l:exponential-homotopy-northcott} 
(cf. Appendix \ref{s:geometry-fundamental-groups-main}). 
In particular, Lemma \ref{l:exponential-homotopy-northcott} can be regarded as an analogous counting tool 
of a Counting Lemma of Minkovski 
frequently used in height theory.  
The second one is a theorem of Green (cf. Theorem \ref{t:green-hyperbolic-embedding}) 
which allow to transfer hyperbolic metric information from the Riemann surface $B$ to the family $\mathcal{A}$. 

\par  

The approach of Parshin in his Theorem \ref{t:parshin-90} 
is based on   
Proposition \ref{p:parshin-diagram-abelian} below which is stated without proof in \cite{parshin-90}. 

\begin{proposition} [Parshin]
\label{p:parshin-diagram-abelian}
Let $W \supset T$ be any finite union of disjoint closed discs in 
$B$ such that distinct points of $T$ are contained in distinct discs.  
Let $B_0\coloneqq B \setminus W$. Let $b_0 \in B_0$ and 
denote $\Gamma = H_1(\mathcal{A}_{b_0}, \Z)$, $G= \pi_1(B_0, b_0)$. 
Assume $A[m] \subset A(K)$ for some integer $m\geq 2$. 
Then   
we have a natural commutative diagram of homomorphisms: 
\begin{equation}
\begin{tikzcd}
\label{d-abelian-family-diagram}
A(K)/mA(K) 
 \arrow[r, hook, "\delta"]  
 &  H^1(\widehat{G}, A[m])  \arrow[dr, "\simeq"] &  \\ 
A(K)  
\arrow[u] \arrow[r, "\alpha"] 
& H^1(G, \Gamma) \arrow[r, "\beta"] 
& H^1(G, A[m]). 
\end{tikzcd}
\end{equation}

\end{proposition}

Here, $\widehat{G}$ denotes as usual the profinite completion of the group $G$. Let $\omega \in W$ be the
finite subset containing the centers of the discs in $W$ and such that $T \subset \omega$. 
By the theory of \' etale fundamental groups, 
$\widehat{G} =    \Gal (K_\omega / K)$ 
where $K_{\omega}/K$ is the maximal Galois extension 
of $K$ which is unramified outside of $\omega$.  
We first indicate below how Proposition~\ref{p:parshin-diagram-abelian} allows us to reduce the problem to  the 
finiteness of certain morphisms between certain 
fundamental groups. In what follows, we
keep the notation as in Proposition~\ref{p:parshin-diagram-abelian}. 
\par 
Since $\mathcal{A}_{B_0} \to B_0$ is a proper submersion, 
it is a fiber bundle by Ehresmann's fibration theorem (cf. \cite{ehresmann}). 
It follows that we have an exact sequence of fundamental groups 
induced by the fiber bundle $A_{b_0}  \to \mathcal{A}_{B_0} \to B_0$ of $K(\pi, 1)$-spaces where we denote $A_{b_0}\coloneqq \mathcal{A}_{b_0}$: 
\begin{equation}
\label{e:abelian-homotopy-exact-sequence-1}
0 \to \pi_1(A_{b_0}, w_0)=H_1(A_{b_0}, \Z) \to \pi_1(\mathcal{A}_{B_0}, w_0) \xrightarrow{\rho \,= f_*} \pi_1(B_0,b_0) \to 0. 
\end{equation}

To fix the ideas, $w_0$ is chosen here and in the rest of the paper to be the 
neutral element of the group $A_{b_0}$, which also lies on the zero section of $\mathcal{A}_{B_0}$.    
\par
Every rational point $P \in A(K)$ induces (cf. \eqref{e:definition-of-i-p-main-reduction-step})  a homotopy conjugacy class of sections  
$i_P \colon  \pi_1(B_0,b_0)  \to  \pi_1(\mathcal{A}_{B_0}, w_0)$ 
of the exact sequence \eqref{e:abelian-homotopy-exact-sequence-1}. 
The quantitative version 
of the homotopy reduction step of Parshin can be stated as follows: 
 
\begin{proposition}
\label{p:homotopy-rational-abelian}
Let the notation be as in Proposition \ref{p:parshin-diagram-abelian}.  
Modulo the translations by the trace $\Tr_{K /\C} A(\C)$, 
every conjugacy class of a section $i$ of the exact sequence 
\eqref{e:abelian-homotopy-exact-sequence-1} 
is induced by at most $t_A=\#  (A(K)/\Tr_{K/ \C} (A) (\C))_{tors} $ rational points $P \in A(K)$. 
\end{proposition}

\begin{proof} (cf.  \cite[Proposition 2.1]{parshin-90})
Let $P, Q \in A(K)$ and assume that  that the conjugacy classes
of $i_P$ and $i_Q$  
of the exact sequence \eqref{e:abelian-homotopy-exact-sequence-1} 
are equal as homomorphisms $  \pi_1(B_0,b_0) \to  \pi_1(\mathcal{A}_{B_0}, w_0) $. 
Up to making a finite base change
$B' \to B$ \' etale outside of $T$, we can suppose without loss of generality that $A[m] \subset A(K)$ 
for some integer $m \geq 2$ (this assumption is only necessary in Proposition~\ref{p:parshin-diagram-abelian}).
It follows that $\alpha(P)=\alpha(Q)$. 
Since $\alpha$ is a homomorphism,  $\alpha(P-Q)=0$. 
Proposition \ref{p:parshin-diagram-abelian} then implies that 
$\delta(P-Q)=0$ and that $P-Q \in mA(K)$. 
Thus, we have $P-Q=mR$ for some $R \in A(K)$. 
Observe that $m\alpha(R)= \alpha(mR)= \alpha(P-Q)=0$ in 
the torsion free abelian group $H^1(G, H_1(A_{b_0}, \Z))$. 
We deduce that $\alpha(R)=0$ since $m \neq 0$.  
Therefore, by an immediate induction, 
the same argument shows that $P-Q \in m^kA(K)$ for every $k \in \N$. 
But since $\Omega \coloneqq  A(K)/\Tr_{K/ \C} (A) (\C)$ is a finitely generated abelian group, 
we must have $P-Q$ mod $\Tr_{K / \C}(A)(\C) \in \Omega_{tors}$ because  $m \geq 2$. As the later set is finite, 
the conclusion follows as $t_A= \# \Omega_{tors}$ by  definition.  
\end{proof}


\begin{proof}[Proof of Proposition \ref{p:parshin-diagram-abelian}] 
\label{another-proof-diagram-parshin-local-system} 
Let $\mathcal{A}_0$ be the restriction of $\mathcal{A}$ over $B_0$. Let $\mathcal{L}= \mathcal{N}_{\sigma_O(B_0)/\mathcal{A}_0}$ be
the complex Lie algebra of $\mathcal{A}_0$,  viewed as a vector bundle over $B_0$.  Identifying the kernel
$\Gamma$ of the relative exponential map $\mathcal{L}\to \mathcal{A}_0$ with $(R^1 f_* \Z)^\vee$, we obtain a canonical short 
exact sequence of locally constant analytic sheaves over $B_0$: 
\begin{equation}
\label{e:exp-relative} 
0 \to (R^1f_* \Z )^\vee \to \mathcal{L} \to \mathcal{A}_0^{an} \to 0. 
\end{equation}
\par
When the trace of $\mathcal{A}$ is zero, we obtain a map $\mathcal{A}_0^{an}(B_0) \to H^1(B_0, (R^1f_*\Z)^\vee)$ whose restriction
to the set of rational sections $A(K) \subset  \mathcal{A}_0(B_0)$ is injective and this is already good
enough for the proof of Proposition~\ref{p:parshin-diagram-abelian}.
\par 

In general, consider the multiplication-by-$m$ $B_0$-morphism $[m] \colon \mathcal{A}_0 \to \mathcal{A}_0$. The induced
map $[m] \colon \mathcal{L} \to \mathcal{L}$ on $\mathcal{L}$ is an isomorphism with inverse $[m^{-1}] \colon  \mathcal{L} \to \mathcal{L}$ given by the multiplication
by $m^{-1}$. Notice that we assume $\mathcal{A}_0[m] \subset  A(K) \subset  \mathcal{A}(B_0)$. The map $[m]$ and the
sequence~\eqref{e:exp-relative}  induce the following commutative diagram in the analytic category:

\begin{equation}
\label{d-abelian-family-diagram-general}
\begin{tikzcd}
 &        
&  
&   0 \arrow[d] 
& \\
 & 0 \arrow[r]  \arrow[d] 
 & 0 \arrow[d] \arrow[r] & \mathcal{A}_0[m]   \arrow[d]  &   \\ 
0 \arrow[r] &  (R^1f_* \Z )^\vee  \arrow[d, "m"] \arrow[r] 
 & \mathcal{L} \arrow[r] \arrow[d, "\simeq "]  
 & \mathcal{A}_0 \arrow[d, "m"]  \arrow[r] &  0  \\ 
0 \arrow[r] &  (R^1f_* \Z )^\vee  \arrow[d] \arrow[r] 
 & \mathcal{L} \arrow[r] \arrow[d,]  
 & \mathcal{A}_0 \arrow[r] \arrow[d]  &  0 \\ 
  & Q \arrow[d]       \arrow[r] 
& 0 \arrow[r] 
& 0  &  
\\
& 0 & & &
\end{tikzcd}
\end{equation}

By the snake lemma, 
$\mathcal{A}_0[m] \simeq Q$. Let $S\supset T$ be the centers of the discs in $W$ and let
$K_S$ be the maximal Galois extension of $K$ unramified outside of $S$.  
Then $\mathrm{Gal}(K_S/K) \simeq \widehat{G}$ where $G = \pi_1(B_0; b_0)$. The natural isomorphism $H^1(\widehat{G}, A[m])\simeq  H^1(G,A[m])$ 
induced by the injection $G \to \widehat{G}$ 
(cf.~\cite[I.2.6.b]{serre:galois-cohomology}  
). The Kummer exact sequence $0 \to A(m) \to A(\bar{K}) \xrightarrow{m} A(\bar{K}) \to 0$ 
 gives rise to an exact sequence of Galois cohomology (thus
of algebraic nature)

\begin{equation}
\label{e:kummer-short-cohomology}   
A(K) \xrightarrow{m} A(K) \to H^1(\widehat{G}, A[m]).
\end{equation}

It is well-known that  
the category of local systems on a Eilenberg-MacLane  
$K(\pi, 1)$-space $X$ (i.e., $\pi_i(X)=0$ for all $i>0$) is equivalent to the category of $\pi$-modules. 
As $\mathcal{A}_0$ and $B_0$ are $K(\pi,1)$-spaces, 
there are canonical 
isomorphisms 
\begin{equation}
    \label{e:auxiliare-isomorphisms}
H^1(B_0,(R^1f_* \Z )^\vee)  \simeq H^1( G, \Gamma), \quad H^1(B_0, \mathcal{A}_0[m]) \simeq H^1(G, A[m]), 
\end{equation} 
where $\Gamma = H_1(A_{b_0}, \Z) \simeq (R^1f_* \Z )^\vee_{b_0}$. 
Here, the group $G$ acts naturally on $\Gamma$ by monodromy (cf.~\cite{phung-19-phd}) 
 and acts trivially on $A[m]$.
 \par 
Combining \eqref{e:kummer-short-cohomology} with the cohomology long exact sequences induced by 
Diagram \eqref{d-abelian-family-diagram-general}, we obtain a natural commutative diagram:  

 \begin{equation}
\label{d-abelian-family-diagram-general-2}
\begin{tikzcd}     
 A(K)
 \arrow[d, "m"] \arrow[r] 
 & 
 \mathcal{A}_0(B_0)  \arrow[r] \arrow[d, "m "]  
 &  H^1(B_0, (R^1f_* \Z )^\vee) 
 \arrow[d, "\simeq"]  \\ 
 A(K)     \arrow[d] \arrow[r] 
 & 
\mathcal{A}_0(B_0)   \arrow[r] \arrow[d]  
 &  H^1(B_0, (R^1f_* \Z )^\vee)  \arrow[d]  \\
  H^1(\widehat{G},  A[m])      \arrow[r, "\simeq"] 
& H^1(B_0, \mathcal{A}_0[m]) \arrow[r] 
& H^1(B_0, Q) 
\end{tikzcd}
\end{equation}

By decomposing the first column into $A(K) /mA(K) \hookrightarrow H^1(\widehat{G}, A[m])$ and using the isomorphisms in \eqref{e:auxiliare-isomorphisms}, we obtain as claimed the 
  commutative diagram 
\eqref{d-abelian-family-diagram}.  
\par 
The detailed descriptions of the involved homomorphisms $\alpha, \beta, \gamma$ and the monodromy action of $G$ are given in \cite[Appendix 8]{phung-19-phd}. 
\end{proof}

\subsection{The homomorphism $\alpha$ and the monodromy action of $G$} 
 \label{monodromy-action-fibre-bundle}
 Keep the notation
be as in Proposition~\ref{p:parshin-diagram-abelian}. 
As $A_{b_0}$ is a torus, 
we can fix a collection of smooth geodesics $l_{w_0,w} \colon [0,1] \to A_{b_0}$
such that $l_{w_0,w} (0)=w_0$ and $l_{w_0,w} (1)=w \in A_{b_0}$. 
Now, each section  $\sigma_P \colon B_0 \to \mathcal{A}_{B_0}$ 
induces naturally a section $i_P \colon \pi_1(B_0,b_0) \to \pi_1(\mathcal{A}_{B_0}, w_0)$ of 
\eqref{e:abelian-homotopy-exact-sequence-1} as follows. 
Take any loop $\gamma$ of $B_0$ based at $b_0$, 
we define the section $i_P$ by the formula: 
\begin{equation}
\label{e:definition-of-i-p-main-reduction-step}
i_P([\gamma])=[ l^{-1}_{w_0,\sigma_P(b_0)} \circ \sigma_P(\gamma) \circ l_{w_0,\sigma_P(b_0)}] 
\in  \pi_1(\mathcal{A}_{B_0}, w_0). 
\end{equation} 
\par 
As a convention, we     concatenate oriented paths as above, 
as opposed to the usual composition of homotopy classes, 
so the multiplication order reverses. 
 As $i_P,i_O$ are sections of $\rho$, the difference $i_P-i_O$ satisfies
$\rho(i_P - i_O) =  \rho(i_P)  - \rho(i_O) =0$. 
Therefore,  $\im (i_P- i_O ) \subset \Ker \rho = H_1(A_{b_0}, \Z)$. 
We have just defined a map $ i_P -i_O \colon \pi_1(B_0,b_0) \to H_1(A_{b_0}, \Z)$.  
This
map is well-defined modulo a principal crossed homomorphism induced by different choices
of the paths  $l_{w_0,w}$ (these choices of paths also give rise to the conjugation class of $i_P$).  
\par
By the exact sequence \eqref{e:abelian-homotopy-exact-sequence-1}, it is not hard to check that 
$i_P -i_O$ is a $1$-cocycle of the group $G= \pi_1(B_0, b_0)$ with coefficients 
in $\Gamma=H_1(A_{b_0}, \Z) \simeq \Z^{2\dim A}$. 
By this way,  we obtain a natural induced natural homomorphism of groups:  
\begin{align}
\label{e:homomorphism-alpha}
\alpha \colon A(K) \to H^1(G, \Gamma), \quad P   \mapsto i_P - i_O, 
\end{align}
where the monodromy $G$-action on $\Gamma$ is given by conjugation as follows.  
Let $\lambda \colon [0,1] \to A_{b_0}$ be a loop  
with $\lambda(0)= \lambda(1)= w_0$. 
Let $\gamma \colon [0,1] \to B_0$ 
be a loop in $B_0$ with $\gamma(0)=\gamma(1)=b_0$. 
Let $\gamma' = \sigma_{O} \circ \gamma$.  
By  \eqref{e:abelian-homotopy-exact-sequence-1}, 
$\gamma' \circ \lambda \circ \gamma'^{-1}$ defines an element in $\pi_1(A_{b_0}, w_0)$ 
denoted $[\gamma] \cdot [\lambda]$. 
It is clear that $[\gamma] \cdot [\lambda]$ depends only on 
the homotopy classes $[\lambda]$ and $[\gamma]$ (with base points). 

\subsection{Some metric properties of hyperbolic manifolds} We follow closely \cite{parshin-90}. 
Let $X$ be a complex manifold. 
The \emph{pseudo Kobayashi hyperbolic metric} $d_X \colon X \times X \to X$ 
is defined as follows. 
Let $\rho$ be the Poincar\' e metric on the unit disc $\Delta= \{z \in \C \colon |z|=1\}$.  
\par
Let $x, y \in X$. Consider the data $L$ consisting of 
a sequence of points $x_0=x, x_1, \dots, x_n=y$ in $X$,  
a sequence of holomorphic maps 
$f_i \colon \Delta \to X$ and the pairs $(a_i, b_i) \in \Delta^2$ for 
$i=0, \dots, n$ such that $f_i(a_i)=x_i$ and $f(b_i)=x_{i+1}$. 
Let $H(x,y; L)= \sum_{i=0}^n \rho(a_i, b_i)$. 

\begin{definition}
[cf. \cite{kobay-hyperbolic-definition}]
\label{pseudo Kobayashi hyperbolic metric} 
For $x, y \in X$, we define $d_X(x,y) \coloneqq \inf_{L}   H(x,y;L)$. 
\end{definition}

If $d_X(x,y)>0$ for all distinct $x, y \in X$, i.e., 
when $d_X$ is a metric, 
 $X$ is called a \emph{hyperbolic} manifold. 
Recall the fundamental distance-decreasing property 
(cf. \cite[Proposition 3.1.6]{kobay-hyperbolic}): 

\begin{lemma}
\label{l:distance-decreasing-hyperbolic}
Let $f \colon X \to Y$ be a holomorphic map of complex manifolds. 
Then for all $x, y \in X$, $d_Y(f(x), f(y)) \leq d_X(x,y)$. 
In particular, if $X \subset Y$, we have $d_Y\vert_X \leq d_X$. 
\end{lemma}

\begin{proof}
For every data $L=\{x_i, f_i, a_i, b_i\}$ 
associated with the points $x, y$, 
we have a data $f(L)= \{f(x_i), f \circ f_i, a_i, b_i\}$  associated with the points $f(x), f(y)$ 
and $H(x,y;L)= H(f(x), f(y); f(L))$. 
The lemma now follows   from the definition. 
\end{proof}

A complex space $X$ is \emph{Brody 
hyperbolic} if there is no nonconstant holomorphic map 
$\C \to X$.

\begin{theorem} [Green]
\label{t:green-hyperbolic-embedding} 
Let $X$ be a relatively compact open subset of a complex manifold $M$.  
Let $D \subset X$ be a closed complex subspace. 
Denote by $\bar{X}, \bar{D}$ the closures of $X$ and $D$ in $M$. 
Assume that $\bar{D}$ and $\bar{X} \setminus \bar{D}$ are Brody hyperbolic. 
Then $X \setminus D$ is hyperbolic and we have  
$d_{X \setminus D} \geq \rho\vert_{X \setminus D}$ 
for some Hermitian metric $\rho$ on $M$. 
In particular, if $M$ is compact and $\lambda$ is any Riemannian metric on $|M|$ then 
there exists $c >0$ such that 
$d_{X \setminus D} \geq c \lambda \vert_{X \setminus D}$. 
\end{theorem}

\begin{proof}
See \cite[Theorem 3]{gre-78}. 
\end{proof}

\begin{theorem}
[Green] 
\label{t:green-hyperbolic-subvariety-abelian} 
Let $X \subset A$ be a complex subspace and let $D$ be a hypersurface of a complex torus $A$. 
Then the following hold: 
\begin{enumerate} [\rm (i)]
\item
 $X$ is hyperbolic if and only if $X$ does not contain any translates of 
a nonzero complex subtorus of $A$;  
\item
if $D$ does not contain any translates of nonzero subtori of $T$ 
then $A\setminus D$ is complete hyperbolic. 
\end{enumerate}

\end{theorem}

\begin{proof}
See \cite[Theorems 1-2]{gre-78}.
\end{proof}

Thanks to the distance-decreasing property of the pseudo-Kobayashi hyperbolic metric, 
we have the following important property of sections (cf.~\cite{parshin-90}).  

\begin{lemma}
\label{l:section-geodesic}
Let $f \colon X \to Y$ be a holomorphic map between complex spaces. 
Suppose that $\sigma \colon Y \to X$ is a holomorphic section. 
Then $\sigma(Y)$ is a totally geodesic subspace of $X$, i.e., for all 
$x, y \in Y$, one has 
$d_Y(x, y)= d_{X} (\sigma(x), \sigma(y))$. 
\end{lemma}

\begin{proof}
It is a consequence of distance-decreasing property of the pseudo-Kobayashi hyperbolic metric 
Lemma \ref{l:distance-decreasing-hyperbolic}: $
d_Y (x, y) = d_Y (f(\sigma(x)), f(\sigma(y))) \leq d_{X} (\sigma(x), \sigma(y)) \leq d_Y (x, y)$.
\end{proof}

\section[Proof of Theorem D]{Growth of generalized integral sections and Proof of Theorem \ref{t:parshin-generic-emptiness-higher-dimension}} 
 \label{s:parshin-generic-emptiness-higher-dimension}

\subsection{Some geometry of Riemann surfaces}  
 
In the proof of Theorem \ref{t:parshin-generic-emptiness-higher-dimension}, we shall need 
the following general auxiliary lemma to control the geometry of a countable closed subset (note that our lemma is more general than the last lemma in \cite{parshin-90}). 

\begin{lemma}
\label{l:refined-cover-countable-closed}
Let $R$ be a closed countable subset of a compact surface  $B$   equipped with a metric $\rho$. 
Let $T\subset B$ be a finite subset. 
Then for every $\varepsilon >0$, 
$R \cup T$ is contained in a finite union $Z$ of disjoint closed discs each of radius at most $\varepsilon$ 
such that $\vol_\rho Z \leq \varepsilon$ and 
such that any two distinct points in $T$ are contained in distinct discs. 
\end{lemma}

\begin{proof}
Write $T= \{ t_1, \dots, t_p \}$ and enumerate $R\cup T= (x_n)_{n \geq 1}$ 
such that $x_1=t_1, \dots, x_p=t_p$.  
Let $\delta= (\varepsilon /c)^{1/2} > 0$ where $c\gg 1$ is some large constant to be chosen later.  
We define by recurrence  $(y_n)_{n \geq 1} \subset R$ such that each $y_n$ is 
contained in a small closed disc $V_n$ of $B$ of radius $r_n < \delta/2^n$. 
For $n=1$, let $y_1\coloneqq x_1=t_1$ and let $D \subset B$ be the closed disc of radius $\delta/2$ centered at $y_1$. 
Since $R$ is countable and $]0, \delta/2[$ is uncountable, 
there exists a closed disc $V_1 \subset D$ of radius $r_1\in ]0, \delta/2[$ also centered at $x_1$ such that 
$\partial V_1 \cap R = \varnothing$ and $V_1 \cap T = \varnothing$. 
Similarly, we can find successively for $i= 2, \dots, p$ 
a closed disc $V_i$  
of radius $r_i \in ]0, \delta/2^i[$  
such that 
$t_i \in V_i$,  $\partial V_i \cap R = V_i \cap T = \varnothing$  
and  $V_i \cap (V_1 \cup \cdots \cup V_{i-1})= \varnothing$.    
\par
For $n=k+1 > p $, let $m \geq 1$ be the smallest integer such that $x_{m} \notin V_1 \cup \cdots \cup V_k$. 
Define $y_{k+1} \coloneqq x_{m}$ and let $V_{k+1} \subset B$ be a closed disc of radius $r_{k+1} \in ] 0,  \delta/2^{k+1}[$ 
centered at $y_{k+1}$ 
such that 
$\partial V_{k+1} \cap R = \varnothing$ and $V_{k+1} \cap (V_1 \cup \cdots \cup V_k) = \varnothing$. 
Such $V_{k+1}$ exists since $R$ is countable. 
By construction,   $R \cup T \subset \cup_{n \geq 1} (V_n \setminus \partial V_n)$. 
As $R \cup T\subset B$ is closed and $B$ is compact, $R\cup T $ is compact. 
Hence, $R\cup T$ is contained in a finite union $Z$ of disjoint closed discs $Z=V_{n_1} \cup \cdots \cup V_{n_q}$. 
It follows that $\{1, \dots, m\} \subset \{n_1, \dots, n_q\}$ since $t_i \in V_i$ 
but $t_i \notin V_j$ for all $i \in \{1, \dots, p\}$ and $j \neq i$. 
Thus, distinct points in $T$ belong to distinct discs of $Z$.  
\par
By compactness of $B$, there exists  $c_0>0$ such that 
for every small enough $r>0$, we have for every point $b \in B$ (cf.~Lemma~\ref{l:length-boundary}) that 
$\vol_\rho D(b, r) \leq c_0 r^2$. 
Therefore, for all $c \gg c_0$:  
\[
\pushQED{\qed} 
\vol Z \leq \sum_{n\geq 1} \vol V_n < c_0 \sum_{n \geq 1} (\delta/2^n)^2  
< c_0 \frac{\varepsilon}{c}  \sum_{n \geq 1} 4^{-n} <  \varepsilon.   \qedhere 
\popQED
\]
\let\qed\relax
\end{proof}

\subsection{Proof of Theorem \ref{t:parshin-generic-emptiness-higher-dimension}}
 We can now give the proof of the main result of the paper. We follow closely the strategy of Parshin \cite{parshin-90}  and carefully make all the arguments effective. 
 
\begin{proof}[Proof of Theorem \ref{t:parshin-generic-emptiness-higher-dimension}] 
Let $n\coloneqq \dim A$. 
 We can clearly suppose that $W$ contains exactly 
$t = \#T$ disjoint closed discs $(W_t)_{t\in T}$  centered at the points of $T$. Define 
\begin{equation}
    \label{e:epsilon-proof-main-theorem-A} 
    \varepsilon \coloneqq  \frac{1}{3} \min (\mathrm{rad}(B,d)), \min \{d(W_u, W_t), u,t \in T, u \neq t\}) 
\end{equation}
where $d(W_u, W_t)$ denotes the $d$-distance between $W_u$ and $W_t$ and $\mathrm{rad}(B,d) >0$ denotes the injectivity  radius of $B$. Since the discs $W_t$'s are disjoint and closed, we have $\varepsilon >0$. 
\par 
Consider the non-hyperbolic locus 
\begin{equation}
\label{d:definition-V-hyperbolic-locus}
V \coloneqq \{ b \in B \colon \DD_b \text{ is not hyperbolic}  \}.
\end{equation}
\par 
Then $V$ is an analytic closed subset of $B$ since hyperbolicity is an analytic open property on the base 
in a proper holomorphic family (cf. \cite{brody-78}). 
Observe that $V \subset Z(\mathcal{A}, \DD)$ by 
Theorem \ref{t:green-hyperbolic-subvariety-abelian} where $Z(\mathcal{A}, \DD)$ 
is defined in Lemma \ref{l:hilbert-dedekind-countabe}. 
Since $\DD_K=D$ does not contain any translates of nonzero abelian subvarieties of $A$, 
it follows that $Z(\mathcal{A}, \DD)$ and thus $V$ are at most countable by Lemma \ref{l:hilbert-dedekind-countabe}. 
We can thus apply Lemma \ref{l:refined-cover-countable-closed} to obtain a finite union $Z_\varepsilon$ of disjoint closed discs each of $d$-radius at most $\varepsilon$ and   
such that $V \subset Z_\varepsilon$. Then it follows
from the definition of $\varepsilon$ in that we can enlarge the discs $W_t$'s  into larger disjoint
closed discs $W'_t$'s 
such that $B'_0 \coloneqq B\setminus (\cup_{t \in T} W'_t) \subset B_0$   is a deformation retract of $B_0$ and
$Z_\varepsilon \subset \cup_{t \in T} W't'$. In  particular, $\pi_1(B_0')= \pi_1(B_0)$.  Clearly, it suffices to prove the theorem for $B'_0 \subset B_0$.  Hence, up to replacing $B_0$ by $B'_0$, we can suppose that $V \subset W$. 
\par
Note that $B_0$ is an unbordered hyperbolic Riemann surface.  
Let $\alpha_1, \dots, \alpha_k$ be a fixed system of  simple generators of 
the fundamental group $\pi_1(B_0, b_0)$ for some $b_0 \in B_0$ 
(cf. Section~\ref{l:primitive-simple-basis-fundamental-group}). 
Notice that $k= \rank (\pi_1(B_0))$.
\par
Fix  a Hermitian metric $\rho$ on $\mathcal{A}$. 
Let $P \in I_s$, i.e., $P \in A(K)$ such that $P$ is an 
$(S, \DD)$-integral section for some $S \subset B$ with $\# \left(S \cap B_0 \right) \leq s$. 

By Theorem \ref{t:collars-for-linear-hyperbolic-bound},  
there exist $b \in B_0$ and  simple loops 
$\gamma_1, \dots,  \gamma_k$ based at $b$
 representing respectively the homotopy classes $\alpha_1, \dots, \alpha_k$ 
 up to a single conjugation by using a fixed
collection of chosen paths $(c_{b_0,b})_{b \in B_0}$ in $B_0$ (cf. Definition~\ref{d:conjugation-path-1}) 
such that  $\gamma_j \subset B_0 \setminus S$ and that 
\begin{equation}
\label{e:strong-uniform-abelian-proof} 
\length_{d_{B_0 \setminus S}} (\gamma_j) \leq L(s + 1)  , \quad  \text{for every } j=1, \dots, k,  
\end{equation} 
for some constant $L >0$ independent of $s$, $S$, $b$ and $P$. 
 \par
 By Theorem \ref{t:green-hyperbolic-subvariety-abelian}  and by our reduction to the case $V \subset W$, 
the varieties $\DD_b$ and $\mathcal{A}_b \setminus \DD_b$ are hyperbolic for every $b \in B_0$. 
As we can always suppose that there is at least one closed disc (of strictly positive radius) in the union $W$, we can assume that 
  $B_0$ is hyperbolic. 
  \par 
Let $h \colon \C \to (\mathcal{A} \setminus \DD)\vert_{B_0 }$ be a holomorphic map. 
The holomorphic map $f \circ h \colon \C \to  B_0$ must be constant since $B_0$ 
is hyperbolic. 
Thus, $h$ factors through $\mathcal{A}_b \setminus \DD_b$ for some $b \in B_0$. 
Since $b \in B_0$, we have $b \notin V$ by the definition of $B_0$. 
Then by the definition of $V$ (cf. \eqref{d:definition-V-hyperbolic-locus}), 
Theorem \ref{t:green-hyperbolic-subvariety-abelian} 
implies that $\mathcal{A}_b \setminus \DD_b$ is hyperbolic. 
It follows that $h$ is constant. 
Therefore, up to enlarging
slightly furthermore $W$ if necessary, we see that the analytic closure of $(\mathcal{A}\setminus D)\vert_{B_0}$ in $\mathcal{A}$ is Brody hyperbolic. 
Similarly, the analytic closure $\DD\vert_{B_0 }$ is also Brody hyperbolic.  
Hence, Theorem \ref{t:green-hyperbolic-embedding} 
 implies that there exists $c>0$ such that 
$d_{(\mathcal{A} \setminus \DD)\vert_{B_0}} \geq c \rho\vert_{(\mathcal{A} \setminus \DD)\vert_{B_0}}$. 
\par
Now, let $\sigma_P \colon B \to \mathcal{A}$ 
be the corresponding section of the rational point $P$. 
Notice that for
every $j \in \{1, \dots, k\}$, we have by the definition of $(S, \DD)$-integral sections that: 
\begin{equation*}
\sigma_P(\gamma_j)  \subset \sigma_P(B_0 \setminus S) 
\subset  (\mathcal{A} \setminus \DD)\vert_{B_0 \setminus S}. 
\end{equation*}
\par 
 It follows that for every $j \in \{1, \dots, k\}$, we have: 
\begin{align*}
\length_\rho (\sigma_P(\gamma_j)) 
& \leq c^{-1} \length_{d_{(\mathcal{A} \setminus \DD)\vert_{B_0}}} 
(\sigma_P (\gamma_j)) 
 & \\
& \leq 
 c^{-1} \length_{d_{(\mathcal{A} \setminus \DD)\vert_{B_0 \setminus S}}} 
(\sigma_P (\gamma_j)) 
& \text{(as } ( \mathcal{A} \setminus \DD)\vert_{B_0 \setminus S} \subset (\mathcal{A} \setminus \DD)\vert_{B_0})  
\\
& = 
c^{-1} \length_{d_{B_0\setminus S}} (\gamma_j)  
&  \text{(by Lemma \ref{l:section-geodesic} )}
\\
& \leq  c^{-1}L(s+1)
& \text{(by \eqref{e:strong-uniform-abelian-proof})} \numberthis \label{e:strong-uniform-abelian-proof-bound-riemannian}
\end{align*}
 \par
Let $\sigma_O$ be the zero section of $\mathcal{A} \to B$. 
Denote $w_0= \sigma_O(b_0) \in  {A}_{b_0}$ where $A_{b_0} \coloneqq \mathcal{A}_{b_0} \subset \mathcal{A}$.   
Recall the following short exact sequence 
associated with the locally trivial fibration $\mathcal{A}_{B_0} \to B_0$  
 (cf. \eqref{e:abelian-homotopy-exact-sequence-1} which follows Proposition \ref{p:parshin-diagram-abelian}):
\begin{equation} 
\label{e:short-exact-hyperbolic-proof-parshin-main-proof}
0 \to  \pi_1(A_{b_0}, w_0) \to \pi_1(\mathcal{A}_{B_0} , w_0) \to \pi_1(B_0, b_0)\to 0
\end{equation}
\par 
The zero section of $\mathcal{A}$ induces a 
section $i_O \colon \pi_1(B_0, b_0) \to  \pi_1(\mathcal{A}_{B_0} , w_0)$ of  
  \eqref{e:short-exact-hyperbolic-proof-parshin-main-proof} 
which in turn induces a 
semi-direct product 
$ \pi_1(\mathcal{A}_{B_0} , w_0) =  \pi_1(A_{b_0}, w_0) \rtimes_{\varphi} \pi_1(B_0, b_0) $. 
Here, $\pi_1(B_0, b_0)$ acts on $\pi_1(A_{b_0}, w_0)$ by conjugation 
(see Section \ref{monodromy-action-fibre-bundle}), 
which is the monodromy action denoted by:   
\begin{equation}
\label{e:semi-product-main-hyperbolic} 
\varphi \colon \pi_1(B_0, b_0) \to \Aut(\pi_1(A_{b_0}, w_0)), \quad \alpha \mapsto \varphi_\alpha.  
\end{equation}

Let $\delta_0$ be the diameter of the analytic closure of 
$\mathcal{A}_{B_0}$ in $\mathcal{A}$ with respect to the metric $\rho$. Let $(c_{b_0, b})_{b\in B_0}$ be the collection of smooth directed paths of bounded $d$-length going from
$b_0$ to every point $b \in B_0$. Then by the compactness of the closure of $\mathcal{A}_{B_0}$, the $\rho$-lengths
of the induced collection of directed smooth paths $\{\sigma_O(c_{b_0, b})\}$ are also uniformly bounded, 
say, by a constant $\delta'_0 >0$.  
\par 
We deduce from \eqref{e:strong-uniform-abelian-proof-bound-riemannian} that the conjugacy class of the section $i_P$ of \eqref{e:short-exact-hyperbolic-proof-parshin-main-proof}   associated with $\sigma_P$ admits a representative (also denoted $i_P$) which sends the basis $(\alpha_j)_{1 \leq j \leq k}$ to some 
homotopy classes in $\pi_1(\mathcal{A}_{B_0}, w_0)$ which admit representative loops of $\rho$-lengths bounded by 
\begin{equation}
\label{e:almost-there-allez-main-proof-hyper} 
H(s) \coloneqq c^{-1}L(s+1)+2(\delta_0+\delta'_0). 
\end{equation}
\par 
The term $2(\delta_0+\delta'_0)$ corresponds to the upper bound for the conjugation induced by
the change of base points of the loops $\sigma_O(\gamma_i)$'s from $\sigma_P(b)$ to $w_0=\sigma_O(b_0)$  using a short
path living in the fiber $A_b$ of  $\rho$-length bounded by $\delta_0$ which goes from $\sigma_P(b)$ to $\sigma_O(b)$ and 
the other path $\sigma_O(c_{b_0, b})$ whose $\rho$-length is bounded by $\delta_0'$ by construction. Notice that one
cannot simply use the path $\sigma_P(c_{b_0,b})$  (and another path from $\sigma_P(b_0)$   to $w_0$) since we have
no control on its $\rho$-length. 
\par
Via the semi-direct product $ \pi_1(\mathcal{A}_{B_0} , w_0) = \pi_1(A_{b_0}, w_0) \rtimes_{\varphi} \pi_1(B_0, b_0)$, 
we can write $i_P(\alpha_j)= (\beta_j, \alpha_j)$ where $\beta_j \in \pi_1(A_{b_0}, w_0)$ for every $j \in \{1, \dots, k\}$. 
\par
As already remarked above, we can replace 
$B_0$ by $B_0 \cup \partial B_0$ without loss of generality. 
Thus, $(\mathcal{A}_{B_0}, \rho)$ can be regarded as a compact Riemannian manifold with boundary. 
Let $\pi \colon \tilde{\mathcal{A}}_0 \to \mathcal{A}_{B_0}$, 
$u \colon \R^{2n} \to A_{b_0}$, and $v \colon \Delta \to B_0$ be the universal covering 
maps. 
We have $ \tilde{\mathcal{A}}_0 \simeq \R^{2n} \times  \Delta $ 
and a commutative diagram where the composition of the top row is $\pi$: 
 \[
\begin{tikzcd}
  \tilde{\mathcal{A}}_0 \simeq \R^{2n} \times \Delta 
   \arrow[r, " u \times \Id"]   
 &  A_{b_0} \times \Delta   \arrow[r]   \arrow[d, "\mathrm{pr_2}",swap] 
 & \mathcal{A}_{B_0} \arrow[d," f_{B_0}"] 
 \\
 & \Delta   \arrow[r, "v"]  & B_0.   
\end{tikzcd}
\]
\par 
The right-most square is the pullback of the 
fiber bundle $\mathcal{A}_{B_0} \to B_0$ over 
the contractible open unit   disc $\Delta$ (in the category of differential manifolds). 
 \par
Fix a point $\tilde{w} = (\tilde{x}_0, \tilde{y}_0) \in \R^{2n} \times \Delta$ in 
the fiber $\pi^{-1}(w_0)$ above $w_0 \in \mathcal{A}_{B_0}$.   
Denote $x_0= u (\tilde{x}_0) = w_0 \in A_{b_0}$ and $y_0=v(\tilde{x}_0)=b_0$.  
\par
Let $j \in  \{1, \dots, k\}$. 
We claim that the deck transformation 
$i_P(\alpha_j)w_0=(\beta_j, \alpha_j)w_0  \in \R^{2n} \times \Delta$ 
is then simply given by 
the couple of deck transformations 
\begin{equation*}
(\varphi_{\alpha_j}(\beta_j) x_0, \alpha_j y_0)  \in \R^{2n} \times \Delta
\end{equation*}
 (cf. the monodromy action $\varphi$ in \eqref{e:semi-product-main-hyperbolic}) 
 with respect to the chosen points $\tilde{w} \in \pi^{-1}(w_0) $, $\tilde{x}_0 \in u^{-1}(x_0)$ and 
$\tilde{y}_0 \in v^{-1}(y_0)$ in the universal cover. 
Indeed, let $\gamma \colon [0,1] \to \mathcal{A}$ the path representing $(\beta_j, \alpha_j)$. 
Let $\tilde{\gamma} \colon [0,1] \to \R^{2n} \to \Delta$ be the lifting 
of $\gamma$ such that $\tilde{\gamma}(0)= (\tilde{x}_0, \tilde{y}_0)= \tilde{w}$. 
Define $\gamma' = (u\times \Id) \circ \tilde{\gamma} \colon [0,1] \to A_{b_0} \times \Delta$. 
Then $\gamma'$ is the lifting 
of $\gamma$ such that $\gamma'(0)= (w_0, \tilde{y}_0)$. 
Note that $\alpha_j$ is represented by $f_{B_0} \circ \gamma' \colon [0,1] \to B_0$. 
By the definition of the monodromy action $\varphi$ 
(via the Homotopy Lifting Property \cite[page 45]{vassiliev-topology}), 
 the homotopy class of $\gamma'$ in $\pi_1(A_{b_0}, w_0)$ 
under the projection $A_{b_0} \times \Delta \to A_{b_0}$ 
is $\varphi_{\alpha_j}(\beta_j)$. 
The claim thus follows.   
 \par
The metric $\rho$ on $\mathcal{A}_{B_0}$ 
pullbacks to  
a geodesic Riemannian metric $\tilde{\rho}$ on $ \R^{2n} \times \Delta \simeq \tilde{\mathcal{A}}_0$. 
The metric $\tilde{\rho}$ induces a 
geodesic Riemannian 
metric 
$d_j= \tilde{\rho}\vert_{\R^{2n} \times \{ \alpha_j y_0 \} }$ 
on $\R^{2n} \times \{ \alpha_j y_0 \}$
for every $j=1, \dots, k$. 
Then $u \colon \R^{2n} \to A_{b_0}$ 
makes $A_{b_0}$ into a compact geodesic Riemannian manifold 
 with the induced metric $d_j$ for every $j=1, \dots, k$. 
By construction and by \eqref{e:almost-there-allez-main-proof-hyper}, 
we find that:  
\begin{equation*}
d_j( \varphi_{\alpha_j}( \beta_j )x_0, \tilde{x}_0) \leq \tilde{\rho}( (\beta_j, \alpha_j)w_0, \tilde{w}) \leq H(s).  
\end{equation*}
\par 
Therefore, as $\pi_1(A_{b_0}, w_0)\simeq \Z^{2n}$ is an abelian group of finite rank, 
Lemma \ref{l:exponential-homotopy-northcott}.(i) and 
Proposition \ref{p:same-growth-geometric-group} 
imply that for every $j=1, \dots, k$, 
there exists a constant $m_j> 1$ independent of $s$ 
such that there are at most $m_j(H(s) +1)^{2n}$ possibilities 
for $\varphi_{\alpha_j}(\beta_j)$ and thus for $i_P(\alpha_j)= (\beta_j, \alpha_j)$ as well. 

\par
Let $m_0 = (\max_{1 \leq j \leq k} m_j)^k  >1 $. 
We deduce that 
the number of (conjugacy classes of) sections $i_P$ of \eqref{e:short-exact-hyperbolic-proof-parshin-main-proof}, 
where $\sigma_P$ is an $(S, \DD)$-integral section, is at most 
$m_0(H(s) + 1)^{2nk}$.  
 We can therefore conclude from Proposition \ref{p:homotopy-rational-abelian} that: 
\begin{equation*}
\# \left( 
I_s \text{ mod } \Tr_{K/ \C}(A)(\C)
\right)
\leq   N (s + 1)^{2nk}, 
\quad \text{ for all }s \geq 0 
\end{equation*} 
where $N = t_A m_0 (c^{-1} L+ 2 (\delta_0+\delta'_0))^{2nk}$ 
and $t_A = \#(A(K) / \Tr_{K/ \C}(A)(\C))_{tors}$. 
\end{proof}

\section[Proof of Corollary A]{Proof of Corollary \ref{c:parshin-generic-emptiness-higher-dimension}}
 \label{section:corollary-emptiness-higher-dimension}

\begin{proof}[Proof of Corollary \ref{c:parshin-generic-emptiness-higher-dimension}] For (i), 
the hypothesis $\Tr_{K / \C}(A)=0$ and Theorem \ref{t:parshin-generic-emptiness-higher-dimension} imply that 
the union of all $(S, \DD)$-integral sections of $A$, where $S \subset B$ 
with $\# S \cap B_0 \leq s$, is finite.  
Let $P_1, \dots, P_q \in A(K)$ be all such integral points where 
$q \leq m(s+1)^{r}$ for   $m=m(\mathcal{A}, B_0)$ given in   Theorem \ref{t:parshin-generic-emptiness-higher-dimension} 
and $r=2 \dim A.  \rank \pi_1(B_0)$. 
For each $i=1 , \dots, q$,   
$f(\sigma_{P_i}(B) \cap \DD) \cap B_0 = S_i$ 
for some finite subset $S_i\subset B_0$ of cardinality at most $s$. 
In particular, for every $i=1, \dots, q$, we have $\sigma_{P_i}(B) \not\subset \DD$ so that 
$\sigma_{P_i}(B) \cap \DD$ is finite as $\sigma_{P_i}$ is algebraic. 
Hence, we can define the  finite intersection loci in $B$ by 
$ E \coloneqq \cup_{i=1}^q f(\sigma_{P_i}(B) \cap \DD)   \subsetneq B$. 
Moreover, we have:  
$$
\# E \cap B_0 = \# \cup_{i=1}^q S_i \leq qs\leq m(s+1)^rs.    
$$
\par 
 We claim that $E$ verifies the point (i). 
Indeed, let $S \subset B \setminus E$ be any subset such that $\# S \cap B_0 \leq s$ and 
suppose on the contrary that $P \in A(K)$ is an $(S,\DD)$-integral point.  
Then $P=P_j$ for some $1 \leq j \leq q$ by the definition of the $P_i$'s. 
Hence, $f( \sigma_{P}(B)\cap \DD) \subset E$. 
But $P$ is $(S, \DD)$-integral so that $f( \sigma_{P}(B)\cap \DD) \subset S$. 
As $S \cap E = \varnothing$, we deduce that $f( \sigma_{P}(B)\cap \DD)= \varnothing$ and thus 
$\deg_B \sigma_P^*\DD=0$. 
This is a contradiction since $\DD$ is strictly nef by hypothesis. 
We conclude the proof of Corollary \ref{c:parshin-generic-emptiness-higher-dimension}.(i).  
\par
For (ii), 
let $E$ and $r$  be given in Corollary \ref{c:parshin-generic-emptiness-higher-dimension}.(i). 
Let $\Delta \subset B^{(s)}$, where $B^{(s)}= B^s/\mathfrak{S}_s$ is the $s$-th symmetric product, 
be the image of the $(s-1)$-dimensional closed subset 
$E \times B^{s-1} \subset B^{s}$ under the  
quotient map $p \colon B^{s} \to B^{(s)}$. 
Since $p$ is a finite morphism of algebraic schemes, 
$\Delta$ is an $(s-1)$-dimensional algebraic 
closed subset of $B_0^{(s)}$. 
Now let $[S] \in B^{(s)} \setminus \Delta$ and let $S= \supp [S] \subset B$. 
Hence $\#S \leq s$ and it follows from the construction of $\Delta$ that $S \cap E= \varnothing$. 
\par 
We claim that none of the $P_i$'s is an $(S,\DD)$-integral point. 
Indeed, if $P_i$ is $(S,\DD)$-integral then $\sigma_{P_i}(B) \cap \DD$ 
is nonempty since $\DD$ is strictly nef by hypothesis.
\par 
On the other hand, $f(\sigma_{P_i}(B) \cap \DD ) \subset S \cap E= \varnothing$ by 
the definition of $E$. We arrive at a contradiction and the claim is proved. 
 The proof of Corollary \ref{c:parshin-generic-emptiness-higher-dimension} is thus completed. 
 \end{proof}
 
\begin{remark}
Corollary \ref{c:parshin-generic-emptiness-higher-dimension}.(ii) can be made  quantitative. 
Let $B_0^{(s)}= B_0^s/\mathfrak{S}_s \subset B^{(s)}$ be the $s$-th symmetric product of $B_0$. 
Let $E_0= E \cap B_0 \subset B$.  
Let $\Delta_0 = \Delta \cap B_0^{(s)}$ 
then $\Delta_0$ the union of  
$\# E_0 \leq ms(s+1)^r$ closed subspaces   
$E_0 \times B_0^{s-1} \subset B_0^{s}$. 
Then  
$\Delta_0$ is an $(s-1)$-dimensional algebraic 
closed subspace of $B_0^{(s)}$. 
Let $[S_0] \in B_0^{(s)} \setminus \Delta_0$ and let $S_0= \supp [S_0] \subset B_0$. 
Then $\#S_0 \leq s$ and it follows that $S_0 \cap E_0= \varnothing$. 
For every $P \in A(K)$, $\sigma_P(B) \cap \DD \neq \varnothing$ as $\DD$ horizontally strictly nef. 
Hence, none of the $P_i$'s is an $(S_0,\DD)$-integral point by the definition of $E$. 
It follows that there 
is no $(S_0,\DD)$-integral points of $\mathcal{A}$ whenever $[S_0] \in B_0^{(s)} \setminus \Delta_0$. 
\end{remark}

\section[Proof of Theorem E]{Proof of Theorem \ref{p:uncountable-limit-point}}
\label{s:uncountable-limit-point}

\begin{proof}[Proof of Theorem \ref{p:uncountable-limit-point}] 
By the Lang-N\' eron theorem, $A(K)$ is a finitely generated abelian group 
since $\Tr_{K /\C}(\mathcal{A}_K)=0$. 
In particular, the subset $R \subset A(K)$ is at most countable. 
On the other hand, for every $P \in R$, the set $\sigma_P(B) \cap \DD$ 
is finite since $P \notin D$. 
It follows that the set $I(R, \DD)$ is at most countable. 
Let $V \coloneqq \{ b \in B \colon \DD_b \text{ is not hyperbolic}  \} \subset B$. 
Then $V$ is an analytically closed subset of $B$  (cf. \cite{brody-78}) which is at most countable 
by Lemma \ref{l:hilbert-dedekind-countabe} since $V \subset Z(\mathcal{A}, \DD)$ by 
Theorem \ref{t:green-hyperbolic-subvariety-abelian} where $Z(\mathcal{A}, \DD)$ 
is defined in Lemma \ref{l:hilbert-dedekind-countabe}. 
\par
For (i), we suppose on the contrary that $I(R, \DD)$ is analytically closed in $B$. 
We fix an arbitrary finite
disjoint union $W$ of closed discs centered at the points of $T$ (points of bad reduction of
the family $\mathcal{A}\to B$).  Then the same argument at the beginning of the proof of Theorem~\ref{t:parshin-generic-emptiness-higher-dimension} implies
that we can enlarge the discs in $W$ to contain $V\cup I(R,D)$ so that they are still closed,
disjoint and contain points of $T$  separately. 
\par 
Then Theorem \ref{t:parshin-generic-emptiness-higher-dimension} applied for 
$B_0= B \setminus W$ tells us that 
$A(K)$ and thus $R \subset A(K)$ contains only finitely many $(W, \DD)$-integral points, 
since $\Tr_{K/ \C} (A)=0$. 
On the other hand, as $I(R, \DD) \subset W$, 
every $P \in R$ is an $(W, \DD)$-integral point. 
It follows that $R$ must be a finite subset, which is a contradiction to the assumption that 
$R$ is infinite. 
Hence, $I(R, \DD)$ is not analytically closed in $B$ and in particular, it must be infinite.  
The proof of (i) is thus completed. 
Observe that the same argument proves also  (iii).  
\par 
The set of limit points $I(R, \DD)_\infty$ is closed in the analytic topology.  
For (ii), suppose also on the contrary that $I(R, \DD)_{\infty}$ is countable. 
Then~(i) implies that $\overline{I(R, \DD)}=I(R, \DD) \cup I(R, \DD)_{\infty}
$ 
is a countable and analytically closed subset of $B$. 
Therefore, the same argument as above shows that 
$R$ is a finite subset of $A(K)$ which is again a contradiction. 
This proves (ii). 
\end{proof} 

\section{Some generalizations} 

By using Theorem~\ref{t:collar-moving-discs-general}, 
the proof of Theorem~\ref{t:parshin-generic-emptiness-higher-dimension} can be easily modified, \emph{mutatis mutandis},
to obtain the following stronger result in which we allow  furthermore the intersection of integral
sections with $D$ to happen over some bounded moving discs in $B_0$: 

\begin{theorem}
\label{t:main-generalization-of-A}
 In Setting (P),  
 let $W \supset T$ be any finite union of disjoint closed discs in 
$B$ such that distinct points of $T$ are contained in distinct discs. Let $B_0 = B \setminus W$ and 
$p\in \N$. Then there exist $m, r > 0$  such that for all $s >0$ and  $I_s^p \coloneqq \cup_{Z}\{ P \in A(K) \colon \# f(\sigma_P (B_0 \setminus Z) \cap \DD_z) \leq s \}$ with the union  over all unions $Z$ of $p$ discs of $d$-radius $r$ in  $B$, one has 
\[
\pushQED{\qed}
\# I_s^p  \text{ mod } \Tr_{K/\C} (A) (\C)
 \leq m(s+1)^{2 \dim A . \rank \pi_1(B_0)}.
\qedhere
\popQED
\]  
\end{theorem}

As an application, we obtain the following generalization of Corollary~\ref{t:lang-vojta-corollary}. The proof 
is similar to that of Corollary~\ref{t:lang-vojta-corollary} but we use Theorem~\ref{t:main-generalization-of-A}  instead of Theorem~\ref{t:parshin-generic-emptiness-higher-dimension}. 

\begin{corollary}
\label{t:lang-vojta-corollary-generalized-1}
In Setting (P), let $W \supset T$ be a finite union of disjoint closed discs
in $B$ such that distinct points of $T$  are contained in distinct discs. Let $B_0=B \setminus W$ and  
$p, s \in \N$. Then there exist $r, M = M(\mathcal{A}, \mathcal{D}, B_0, p,s) >0$ such that for every union $Z$ of $p$ discs of $d$-radius $r$ in
$B$ and every section
$\sigma \colon B \to \mathcal{A}$  with  $\#f(\sigma(B_0\setminus Z) \cap \mathcal{D}) \leq s$, one has $\deg_B \sigma^*\mathcal{D} < M$. 
\qed 
\end{corollary}

\appendix 

\section{Geometry of the fundamental groups} 
\label{s:geometry-fundamental-groups-main}
We concisely collect standard results on the geometry of 
the fundamental groups which are necessary for 
the proof of Theorem \ref{t:parshin-generic-emptiness-higher-dimension}. 
For the convenience of future works, we give 
the main statements Proposition \ref{p:same-growth-geometric-group} 
and Lemma \ref{l:exponential-homotopy-northcott} which are 
slightly more general than what we actually need. 
We first recall the following notion of quasi-isometry introduced by Gromov: 
 
\begin{definition}
[Gromov] 
Let $X, Y$ be metric spaces. 
Let $L, A >0$. We say that a map $f \colon X \to Y$ (not necessarily continuous) 
is $(L,A)$-quasi-isometry if: 
\begin{enumerate}
\item
(equivalence)  $ \frac{1}{L} d(x,y) -A  \leq d(f(x),f(y)) \leq Ld(x,y) +A$,  for all  $x, y \in X$;
\item
(quasi-surjective) there exists $R >0$ such that 
$f(X) \cap B(y,R) \neq \varnothing$ for every $y \in Y$  where $B(y,R)= \{z \in Y \colon d(z,y)<R\}$ is the ball of radius $R$ centered at $y$.   
\end{enumerate}
   
\end{definition}

Note that being \emph{quasi-isometric to} is an equivalence
relation. 

\begin{definition}
 Let $G$ be a finitely generated group and let $S$ be a finite generating subset. 
 The function $d_S  \colon  G \times G \to  \N$ given by
$$
(g, h) \mapsto 
\left\{
	\begin{array}{ll}
		0  & \mbox{if } g = h \\
		\min \{ n \colon g^{-1} h = s_1 \cdots s_n \text{ for some }s_1, \dots s_n \in S \cup S^{-1} \} & \mbox{otherwise } 
	\end{array}
\right.
$$
defines a metric on $G$ called the \emph{word metric} associated with the generating set $S$. We denote also 
$B((G, d_S), R) = \{g \in G \colon d_S(g,1_G) <R\}$ or simply $B(G, R)$ if $d_S$ is fixed.
\end{definition}

The following sufficient condition of quasi-isometry is due to Milnor and Svarc. 
It is sometimes called the Fundamental Lemma of Geometric
Group Theory \cite[Proposition~I.8.19]{bridson-haefliger}. 

\begin{theorem}
[Milnor-Svarc] 
\label{t:milnor-svarc}
Let $(X,d)$ be a geodesic metric space and $G$ a finitely generated group acting on $X$ by isometries, i.e., 
$ d(gx,  gy)= d(x, y) $. Let $x_0\in X$  and assume:
\begin{enumerate} [\rm (i)]
\item
(cobounded action) there exists $R > 0$ such that  translates of $B(x_0, R)$
cover $X$, i.e., 
$$
X= \cup_{g \in G} (gB(x_0, R)) = \cup_{g \in G} B(gx_0, R); 
$$ 
\item
(metrically proper action) for any $r > 0$,  
$
\{ g \in G \colon B(x_0,  r) \cap g B(x_0, r) \neq \varnothing\}$, is finite. 
\end{enumerate} 
\par 
Then $G$ is finitely generated and the map $p \colon G \to  X$,  $p(g) = gx_0$, is a quasi-isometry where $G$ 
is equipped with an arbitrary word metric associated with a finite system of generators. \qed
\end{theorem}
 
\begin{lemma}
\label{l:at-most-exponental-grp}
Every finitely generated group $G$ has at most exponential growth, i.e., 
for every finite generating subset $S \subset G$, there exists $N >1$ such that 
for every $R>0$, we have 
\begin{equation}
\label{e:at-most-exponential}
\# B((G, d_S), R) \leq N^R.  
\end{equation}

\end{lemma}

\begin{proof}
We will show that every $N \geq 2s +1$ satisfies the 
inequality \eqref{e:at-most-exponential}, where $s= \#S$. 
Moreover, the equality holds only if $G$ is a free group of finite rank. 
Indeed, by the definition of the word metric $d_S$,  we have for every $R >0$ that  
$ B((G, d_S), R) \subset \left\{ \prod_{j=1}^{\floor{R}} s_j \colon s_j \in S \cup S^{-1} \cup \{1_G\} \right\}
$ where $\floor{R}$ denotes the largest integer smaller than or equal to $R$. 
It follows that 
\begin{align*}
\#B((G, d_S), R) & \leq \# \left\{ \prod_{j=1}^{\floor{R} } s_j \colon s_j \in S \cup S^{-1} \cup \{1_G\} \right\} 
& \leq (\# S + \# S^{-1} + 1)^{\floor{R}} 
& \leq (2s+1)^R
\end{align*}
and the lemma is proved by taking $N= 2s+1$. 
\end{proof}

We mention here the famous theorem of Gromov classifying   groups of polynomial
growth.

\begin{theorem}
[Gromov] 
\label{t:gromov-polynomial-growth}
Let $G$ be a finitely generated group. 
Fix a finite generating system $S \subset G$ and consider the 
corresponding word metric $d_S$.  
Let $n(L)$ be the number of elements $g \in G$ such that 
$d_S(g, 1_G) \leq L$. 
Then there exist $a, r > 0$ such that $ 
n(L) \leq a(L+1)^r$ for all $L >0$ 
if and only if $G$ is virtually nilpotent, i.e., 
$G$ admits a finite index subgroup $H$ which is nilpotent, 
i.e., $H_m=0$ for some $m\geq 0$ where 
$H_{k+1} \coloneqq [H_k, H]$, $H_0=H$.  
\end{theorem}

\begin{proof}
See \cite[Main Theorem]{grom-81}. 
\end{proof}

We recall without proof the following standard theorem. 

\begin{theorem}
 \label{t:pi-1-finite-generation}
Let $M$ be a compact 
Hausdorff space which admits a finite  
cover by open simply connected sets, and which is locally path connected (i.e., there is a
base for the topology consisting of path connected sets). 
Then $\pi_1(M)$ is 
finitely generated.
In particular, this holds for all compact semi-locally simply connected spaces and thus for
all compact Riemannian manifolds with or without boundary.
\end{theorem}

\begin{definition}
[Equivalence of growth functions] 
Two increasing functions $f, g  \colon  \R_+  \to   \R_+$  are said to have the \emph{same order of growth} 
if there exist constants $a, c > 0$ such that for all $r >0$, we have  
$
f(r) \leq c g(ar)$ and  $g(r) \leq cf(ar)$. 
\end{definition}

Hence, if $f, g$ have the same order of growth, then 
$f$ is bounded from above (resp. from below) by a polynomial of degree $m$ if and 
only if so is $g$. 
The main statement is the following. 

\begin{proposition}
\label{p:same-growth-geometric-group}
Let $(M, d)$ be a compact connected Riemannian manifold with 
boundary. Let $\pi \colon \tilde{M} \to M$ be the universal cover of $M$.    
Then for any point $x_0 \in M$ not lying on the boundary, the map
$p \colon \pi_1(M, x_0) \to \pi^{-1}(x_0)$ 
given by $g  \mapsto  g x_0$ is a bijective quasi-isometry. 
In particular, the growths of $\pi_1(M, x_0)$ 
are the same whether calculated with respect to the induced geometric
norm by $d$ on $\pi_1^{-1}(x_0)$ or with respect to the algebraic word norm associated with an arbitrary
finite system of generators of $\pi_1(M, x_0)$. 
\end{proposition}

\begin{proof}
By Theorem \ref{t:pi-1-finite-generation}, $\pi_1(M)$ 
is a  finitely generated group. 
Note that $M$ is also a geodesic metric space. 
Fix $\tilde{x}_0 \in \pi^{-1}(x_0)$. 
Theorem \ref{t:milnor-svarc} 
implies that the deck transformation action of   $\pi_1(M, x_0)$ 
on the universal cover 
$\tilde{M}$  
induces a quasi-isometry between ($\pi_1(M, x_0), d_{word})$ 
and $(\tilde{M}, \tilde{d})$, say, an $(L,A)$-quasi-isometry. 
Here $d_{word}$ denotes the word metric of the group 
$\pi_1(M,x_0)$ associated with some finite generating set and $\tilde{d}$ denotes the induced metric of $d$ 
on $\tilde{M}$. 
Since the action of $\pi_1(M, x_0)$ commutes with  $\pi \colon \tilde{M} \to M$, the same quasi-isometry $g \mapsto gx_0$ 
gives us a bijective $(L,A)$-quasi-isometry between 
$(\pi_1(M, x_0), d_{word})$ and the  fiber $\pi^{-1}(x_0)$ 
equipped with the induced metric $\tilde{d}\vert_{\pi^{-1}(x_0)}$: 
\begin{equation*}
L^{-1}d_{word}(g, h) - A  \leq \tilde{d}\vert_{\pi^{-1}(x_0)} (gx_0, hx_0) \leq Ld_{word}(g, h) + A, \quad \text{for every } 
g, h \in \pi_1(M, x_0). 
\end{equation*}
Hence, for every $D > 0$, we have:  
\begin{align*}
\# B(\pi_1(M, x_0), L^{-1}D + A) 
\leq 
\# B((\pi^{-1}(x_0), \tilde{x}_0), D) 
\leq 
\# B(\pi_1(M, x_0), LD + LA). 
\end{align*}

Here, $B(\pi_1(M, x_0), r)$ 
denotes the ball of $d_{word}$-radius $r$ 
in $\pi_1(M, x_0)$ centered
at $0$. 
Similarly, 
$B((\pi^{-1}(x_0), \tilde{x}_0), r)$ denotes the ball of $\tilde{d}\vert_{\pi^{-1}(x_0)}$-radius 
$r$ in $\pi^{-1}(x_0)$ centered at $\tilde{x}_0$. 
Since $LD + $A and $LD + LA$ are fixed linear functions in $D$, the growths of $\pi^{-1}(x_0)$ and of 
$\pi_1(M, x_0)$ are clearly of the same order.  
\end{proof}

An important application of Proposition \ref{p:same-growth-geometric-group} 
is the following lemma which can be seen as 
the analogue of the Minkovski Counting Lemma: 

\begin{lemma} 
\label{l:exponential-homotopy-northcott} 
Let $(M,d)$ be a connected compact Riemannian manifold with  boundary. 
Let $x_0 \in M$. 
For every $L >0$, define $n(L)$ to be the number of homotopy classes in $\pi_1(M, x_0)$ which admit 
some representative loops of length at most $L$ with respect to the metric $d$. 
Then:  
\begin{enumerate} [\rm (i)]
\item
if $\pi_1(M)$ is virtually nilpotent, there exist $a, r > 0$ such that 
$n(L) \leq a(L+1)^r$   for all $L >0$. 
 If $\pi_1(M)$ is abelian, $r$ can be chosen to be the rank of $\pi_1(M)$;  
\item
In general, there exists $p > 0$ such that 
$n(L) \leq \exp(p(L+1))$  for all $L >0$.   
\end{enumerate}

\end{lemma}

\begin{proof}
Let $\pi \colon \tilde{M} \to M$ be the universal cover of 
$M$ and fix $\tilde{x}_0 \in \pi^{-1}(x_0)$. 
Then $\tilde{M}$ is a connected 
geodesic Riemannian manifold with metric $\tilde{d}$ which makes $\pi$ into a local isometry. 
The length of the loops in $(M,x_0)$ is the same as the length of their lifts in $\tilde{M}$. 
The geometric metric on $\pi^{-1}(x_0)$ is induced by $\tilde{d}$. 
Hence, a loop $\gamma$ in $M$ based at $x_0$  
of length at most $L$ is uniquely determined by 
a point $[\gamma]x_0 \in \pi^{-1}(x_0)$ such that 
$\tilde{d} ([\gamma]x_0, \tilde{x}_0) \leq L$. 
By Theorem \ref{t:milnor-svarc}, $\pi_1(M,  x_0)$ 
is quasi-isometric to the fiber $\pi^{-1}(x_0)$. 
Hence, they have the same order of growth (Proposition \ref{p:same-growth-geometric-group}) 
where the metric on $\pi^{-1}(x_0)$ 
is induced by $\tilde{d}$ while  
the one on $\pi_1(M, x_0)$ is the word metric 
with respect to a finite system of generators (which exists since $\pi_1(M, x_0)$ 
is  finitely generated by Theorem \ref{t:pi-1-finite-generation}). 
Since the growth of $\pi_1(M, x_0)$ 
  is at most exponential (cf. Lemma \ref{l:at-most-exponental-grp}), 
the point (ii) is proved. 
The above discussion and Theorem \ref{t:gromov-polynomial-growth} of Gromov   imply the point (i) 
except for the second statement. 
If $\Gamma= \pi_1(M, x_0)$ is abelian then  $r= \rank \Gamma \geq 0$ is finite (by Theorem \ref{t:pi-1-finite-generation}). 
Hence, $\Gamma = \Z g_1 \oplus  \cdots \oplus \Z g_r  \oplus \Gamma_{tors}$ 
for some $g_1, \dots, g_r \in \Gamma$.  
Recall the word metric $d_S$ on $\Gamma$ where $S= \{ g_1, \dots, g_r\}$. 
It is easy to see that: 
\begin{align*} 
\#\{ g \in \Gamma \colon d_S(g, 1_\Gamma) \leq L \} 
& \leq    \# \Gamma_{tors} .   \#  
\left\{g= \sum_{j=1}^r n_j g_j \colon -L \leq n_j \leq L \right \}   \leq \# \Gamma_{tors} (2L+1)^r. 
\end{align*}
The conclusion follows by Proposition \ref{p:same-growth-geometric-group}. 
\end{proof}

\section{Hilbert schemes of algebraic groups} 

By the theory of Hilbert schemes of subvarieties developed by Grothendieck and 
Altmann-Kleiman, we have the following properties needed in the proof of Theorem \ref{t:parshin-90} and 
Theorem~\ref{t:parshin-generic-emptiness-higher-dimension}. 

\begin{lemma}
\label{t:hilbert-group-schemes}
Let $\pi \colon G\to S$ be a (quasi-)projective group scheme over a scheme $S$. 
Let $\DD \subset G$ be a closed subscheme and let $\VV= G \setminus \DD$. 
Consider the contravariant functors 
$
\FF_{G/S} \colon \Sch_S \to \Ensemble 
$
defined for $T \to S$ an $S$-scheme by 
\begin{equation*}
\FF_{G/S}(T) \coloneqq 
\left \{ 
\begin{aligned}
&\text{ (connected) group subschemes of } G_T, \text{ that are } \\
&\text{ flat, proper, and of finite presentation over } T 
\end{aligned}
\right \}. 
\end{equation*} 
Then the following holds:  
\begin{enumerate} [\rm (a)]
\item
$\FF_{G/S}$ is representable by a locally of finite type $S$-scheme denoted also by 
$\FF_{G/S}$; 
\item
There exist natural immersions of $S$-schemes
$$
\FF_{G/S}, \, \Mor_{S}(S, G), \,  \Hilb_{\DD/S}^{\dim > 0},\,   \Hilb_{\VV/S}^{\dim > 0} \subset \Hilb_{G/S};   
$$ 
where $\Hilb_{X/S}^{\dim >0}$ denotes the 
complement in $\Hilb_{X/S}$  
of the $S$-relative Hilbert schemes of points, i.e., of zero dimensional closed subschemes, 
of an $S$-scheme $X$.  
\item
The $S$-scheme 
$
\FF_{G/S, \DD} \coloneqq \Hilb_{\DD/S}^{\dim >0} \times_{\Hilb_{G/S}} \left(  \FF_{G/S} \times_S \Mor_S(S,G) \right) 
$  
represents the contravariant functor $\Sch_S \to \Ensemble$ given by 
\[
T \mapsto  
\left \{ 
\begin{aligned} 
 & \text{ translates of }
 (\varphi \colon H \to G_T) \in \FF_{G/S} (T) \text{ by a $T$-section } \sigma \colon T \to G_T \\
 & \text{ such that } 
 \dim H >0 \text{ and } 
 \im (\sigma . \varphi) \subseteq  \DD_T 
\end{aligned}
\right \}, 
\]
where $\sigma.\varphi (h) \coloneqq \sigma(\pi_T \circ \varphi(h)) \varphi(h)$ for every $h \in H$. 
That is, $\FF_{G/S, \DD}$ is the moduli space of translates of positive dimensional group subschemes 
of $G_s$ contained in $\DD_s$ for $s \in S$;  
\item
Similarly, the $S$-scheme 
$
\FF_{G/S, \VV} \coloneqq \Hilb_{\VV/S}^{\dim >0} \times_{\Hilb_{G/S}} \left(  \FF_{G/S} \times_S \Mor_S(S,G) \right) 
$  
is the moduli space of translates of positive dimensional group subschemes 
of $G_s$ that have empty intersection with $\DD_s$ for $s \in S$; 
\item
The schemes $\FF_{G/S, \DD}$ and $\FF_{G/S, \VV}$ have only countably many irreducible components. 
\end{enumerate}
\end{lemma}

\begin{proof}
For (a), see \cite[Expos\' e XI, Remarque 3.13]{sga3-II}.  
The existence of other schemes in (b) is standard since $G/S$ is quasi-projective and so are $\DD/S$, $\VV/S$ as 
$\DD \subset G$ is assumed to be closed.  
Observe that $\Hilb_{G/S}$ 
is the disjoint union of quasi-projective $S$-schemes $\Hilb_{G/S}^{q(x)}$ with $q(x) \in \Q[x]$ runs over 
all numerical  
polynomials of degree $\leq \dim G_s$ where $G_s$ is a general fiber. 
Moreover, the $S$-schemes $\FF_{G/S}, \Mor_{S}(S, G) $, $ \Hilb_{\DD/S}^{\dim > 0}$ and 
$\Hilb_{\VV/S}^{\dim > 0}$ are also stratified by the same Hilbert polynomials and 
that $\Hilb_{\DD/S}^{\dim > 0}$ and $\Hilb_{\VV/S}^{\dim > 0}$ do not take into account zero degree polynomials. 
Since each $S$-scheme of finite type has finitely many irreducible components, 
it follows from the stratification by Hilbert polynomials that all 
Hilbert schemes in the lemma have only countably many irreducible components. 
In particular, this proves (e). 
The assertions (c) and (d) are also clear. 
\end{proof}

\begin{lemma}
\label{l:hilbert-dedekind-countabe}
Let $B$ be a Dedekind scheme with fraction field $K$. 
Suppose that $\pi \colon G \to B$ is a quasi-projective $B$-group scheme. 
Let $\DD \subset G$ be a closed subset with generic fiber $D= \DD_K \subset G_K$. 
Then: 
\begin{enumerate} [\rm (i)]
\item
if $D$ does not contain  
any translates of positive dimensional $\overline{K}$-algebraic subgroups of $G_K$, 
then the following exceptional set is countable 
\[
\begin{aligned}
Z(G,\DD) \coloneqq  \left\{  
 t \in B \text{ closed point}\colon \exists x \in G_t, \exists\,H \subset G_t \text{ a subgroup}, \, \dim H > 0, \, 
xH \subset  \DD_t 
\right\} ; 
\end{aligned}
\]

\item
if $G_K \setminus D$ does not contain  
any translates of positive dimensional $\overline{K}$-algebraic subgroups of $G_K$, 
then the following exceptional set is countable 
\[
\begin{aligned}
U(G,\DD) \coloneqq  \left \{  
t \in B \text{ closed point} \colon \exists x \in G_t,\, \exists H \subset G_t \text{ a subgroup}, \, \dim H > 0, \, 
xH \subset  G_t \setminus \DD_t 
\right\}.
\end{aligned}
\]
 
\end{enumerate}

\end{lemma}

\begin{proof}
For (i), let $X$ be an irreducible component of the Hilbert $B$-scheme 
$\FF_{(G/B), \DD}$ defined in Lemma \ref{t:hilbert-group-schemes}. 
Then $X$ is a scheme of finite type over $B$. 
We claim that the induced morphism $f_X \colon X \to B$ 
is not dominant. 
The point (i) will then be proved since $\FF_{(G/B), \DD}$ does not 
dominate $B$. Indeed, 
as $\FF_{(G/B), \DD}$ contains only countably many irreducible components each of which being of finite type over $B$ 
(cf. Lemma \ref{t:hilbert-group-schemes}.(c)), 
the image of $\FF_{(G/B), \DD}$ in $B$ is thus a countable subset of closed points of $B$ 
by Chevalley's theorem 
and the conclusion follows. 
\par
Suppose on the contrary that $X$ dominates $B$. 
Then we can find a 1-dimensional irreducible closed subscheme $C \subset X$ such that 
$C$ dominates $B$. 
Up to replacing $X$ by $C$, we can thus assume that $\dim X =1$. 
Let $\eta$ be the generic point of $X$ and let 
$V \hookrightarrow \FF_{(G/B), \DD} \times_B G$ be the universal group scheme. 
Then $L=  \kappa(\eta) \subset \overline{K}$ is a finite extension of $K$. 
On the other hand, it follows from the functorial property of Hilbert schemes 
that $V_L$ is the universal group scheme of $G_L /L$ avoiding $\DD_L$ 
but $\FF_{(G_L/L), \DD_L}$ is empty by the hypothesis of (i). 
This contradiction shows that $X$ cannot dominate $B$ and (i) is proved. 
The proof of (ii) is similar. 
\end{proof}

\bibliographystyle{siam}

\end{document}